\documentclass[12pt, reqno]{amsart}
\usepackage{amsmath, amsthm, amscd, amsfonts, amssymb, graphicx, color, mathrsfs}
\usepackage[bookmarksnumbered, colorlinks, plainpages]{hyperref}
\usepackage[all]{xy} 
\usepackage{slashed}

\newtheorem{theorem}{Theorem}[section]
\newtheorem{lemma}[theorem]{Lemma}

\newtheorem{proposition}[theorem]{Proposition}
\newtheorem{corollary}[theorem]{Corollary}

\theoremstyle{definition}
\newtheorem{definition}[theorem]{Definition}
\newtheorem{example}[theorem]{Example}
\newtheorem{assumption}[theorem]{Assumption}

\theoremstyle{remark}
\newtheorem{remark}[theorem]{Remark}
  \numberwithin{equation}{section}

\begin{document}
\setcounter{page}{1}

\title[Estimates for sums of eigenfunctions of elliptic operators]{Estimates for sums of eigenfunctions of elliptic pseudo-differential operators on  compact Lie groups}

\author[D. Cardona]{Duv\'an Cardona}
\address{
  Duv\'an Cardona:
  \endgraf
  Department of Mathematics: Analysis, Logic and Discrete Mathematics
  \endgraf
  Ghent University, Belgium
  \endgraf
  {\it E-mail address} {\rm duvanc306@gmail.com, duvan.cardonasanchez@ugent.be}
  }
  
  \author[J. Delgado]{Julio Delgado}
\address{
  Julio Delgado:
  \endgraf
  Departmento de Matem\'aticas
  \endgraf
  Universidad del Valle
  \endgraf
  Cali-Colombia
    \endgraf
    {\it E-mail address} {\rm delgado.julio@correounivalle.edu.co}
  }

\author[M. Ruzhansky]{Michael Ruzhansky}
\address{
  Michael Ruzhansky:
  \endgraf
  Department of Mathematics: Analysis, Logic and Discrete Mathematics
  \endgraf
  Ghent University, Belgium
  \endgraf
 and
  \endgraf
  School of Mathematical Sciences
  \endgraf
  Queen Mary University of London
  \endgraf
  United Kingdom
  \endgraf
  {\it E-mail address} {\rm michael.ruzhansky@ugent.be, m.ruzhansky@qmul.ac.uk}
  }

\thanks{The authors are supported  by the FWO  Odysseus  1  grant  G.0H94.18N:  Analysis  and  Partial Differential Equations and by the Methusalem programme of the Ghent University Special Research Fund (BOF)
(Grant number 01M01021). Julio Delgado is also supported by  Vic. Inv Universidad del Valle. Grant No. CI-7329,  MathAmSud and Minciencias-Colombia under the project MATHAMSUD 21-MATH-03. Michael Ruzhansky is also supported  by EPSRC grant 
EP/R003025/2.
}


\begin{abstract} We extend the estimates proved by Donnelly and Fefferman and by Lebeau and Robbiano for sums of eigenfunctions of the Laplacian (on a compact manifold) to estimates for sums of eigenfunctions of any positive and elliptic pseudo-differential operator of positive order on a compact Lie group. Our criteria are imposed in terms of the positivity of the corresponding matrix-valued symbol of the operator. As an application of these inequalities in the control theory, we obtain the null-controllability for diffusion models for elliptic pseudo-differential operators on compact Lie groups.
\end{abstract} 
\maketitle
\tableofcontents
\allowdisplaybreaks
\section{Introduction}
\subsection{Outline}
Let $(M,g)$ be a compact $C^\infty$-Riemannian manifold.
In the late 1980 H. Donnelly and C. Fefferman in their celebrated {\it Inventiones' paper} \cite{DonnellyFefferman} proved the doubling property
\begin{equation}\label{Fefferman-Donnelly}
    \sup_{B(2R)}|\phi|\leq e^{C_1\lambda +C_2} \sup_{B(R)}|\phi|
\end{equation}for any eigenfunction of the Laplacian $\Delta_g$ on $M,$ that is, $-\Delta_g \phi=\lambda^2 \phi,$ where $B(2R)$ and $B(R)$ represent concentric balls (associated to the geodesic distance) where the constants $C_1$ and $C_2$ are independent of $R>0,$ and depending only on $M.$  The estimate in \eqref{Fefferman-Donnelly} remains valid for sums of eigenfunctions of $\Delta_g.$  In this work we extend such an estimate for sums of eigenfunctions of any positive elliptic pseudo-differential operator $A$ when $M$ is a compact Lie group.    Even, we consider the general case where $A$ has positive real order and belongs to the global $(\rho,\delta)$-H\"ormander classes. 

If this inequality holds in the complete range $0\leq \delta<\rho\leq 1$ was an open problem prior to this work. The fact of considering the setting of compact Lie groups is justified since on general compact manifolds the principal symbol of a pseudo-differential operator is invariantly defined only if $0\leq \delta<\rho\leq 1$ and $\rho\geq 1-\delta,$ see H\"ormander \cite{Hormander1985III}.  Here, we introduce a new approach, different from the one via Carleman estimates as developed by Donnelly and Fefferman in \cite{DonnellyFefferman}. A reason to  introduce a new approach comes from the lack of Carleman estimates in the case of non-local operators. To do this, looking for criteria on the operator $A$ that allow the validity of the doubling property in \eqref{Fefferman-Donnelly} for the sums of its eigenfunctions, we connect this problem  with the  representation theory of a compact Lie group $G.$ In order to present our main Theorem \ref{Main:theorem} let us introduce the required notation.
\subsection{Main result}
Indeed, by writing the elliptic operator $A:C^\infty(G)\rightarrow C^\infty(G)$ in the convolution form 
\begin{equation}
    Af(x)=\int\limits_{G}R_{A}(x,xy^{-1})f(y)dy,\,\,f\in C^\infty_0(G),
\end{equation}where the distribution $R_A\in C^\infty(G, \mathscr{D}'(G))$ is associated via the Schwartz kernel theorem, one can associate a global symbol $\sigma_A:G\times \widehat{G}\rightarrow \cup_{\ell\in \mathbb{N}}\mathbb{C}^{\ell\times \ell}$ to $A.$ Here, $\widehat{G}$ denotes the unitary dual of $G,$ formed by the set of all continuous, unitary and irreducible representations $\xi:G\rightarrow\textnormal{Hom}(\mathbb{C}^\ell)$ of $G.$ Recall that $d_\xi:=\ell$ is usually called the dimension of the representation $\xi$. Then, the global symbol $\sigma_A$ of $A$ is defined by the group Fourier transform of the distribution $K_{A}(x,\cdot),$ which is given by
\begin{equation}
    \sigma_A(x,\xi)=\int\limits_{G}R_{A}(x,z)\xi(y)^*dy,\,\,[\xi]\in \widehat{G}.
\end{equation}Then, the Fourier inversion formula allows the Fourier representation of the operator $A$ as follows
\begin{equation}
    Af(x)=\sum_{[\xi]\in \widehat{G}}d_\xi\textnormal{Tr}[\xi(x)\sigma_A(x,\xi)\widehat{f}(\xi)],\,f\in C^\infty(G),
\end{equation}with $\widehat{f}(\xi)=\int\limits_{G}f(y)\xi(y)^*dy$ denoting the Fourier transform of a test function $f$ at the representation $\xi.$ Above $dy$ denotes the Haar measure on $G.$ This quantisation was consistently developed in \cite{Ruz}, and we recall some of its relevant properties.

The H\"ormander classes of pseudo-differential operators $\Psi^m_{\rho,\delta}(G)$ can be characterised in terms of the global matrix-valued symbols $\sigma_A$ obtained in the construction above. That $A\in \Psi^m_{\rho,\delta}(G) $ means that in any local coordinate system the operator has the form (by identifying local coordinates systems  in $G$ with the corresponding Euclidean open subsets)
$$ A\phi(x) =\int\limits_{\mathbb{R}^n}e^{2\pi i x\cdot \theta}\sigma(x,\theta)\widehat{\phi}(\theta)d\theta,\,\,\phi\in C^\infty_0(\mathbb{R}^n),$$ where $\widehat{\phi}$ denotes the Euclidean Fourier transform of $\phi$ and where the symbol $\sigma$ associated to each chart satisfies growing estimates of $(\rho,\delta)$-type, that is
\begin{equation}
    |\partial_x^\beta\partial_\theta^\alpha \sigma(x,\theta)|\leq C_{\alpha,\beta}(1+|\theta|)^{m-\rho|\alpha|+\delta|\beta|}
\end{equation}uniformly on compact subsets of the  chart.
Indeed, there are required certain relations between $\rho$ and $\delta$, to have the classes $\Psi^m_{\rho,\delta}(G)$ invariant under changes of coordinates, namely that
\begin{equation}\label{rho:delta:restr}
    0\leq \delta<\rho\leq 1,\, \rho\geq 1-\delta.
\end{equation}
Under this assumptions it was proved in \cite{RuzhanskyTurunenWirth2014}, that $A\in \Psi^m_{\rho,\delta}(G)$ if and only if its matrix-valued symbol $\sigma_A$ satisfies the symbol estimates
\begin{equation}\label{RTHormanderclasses}
    \Vert \partial_x^\beta\mathbb{D}^\alpha\sigma_A(x,\xi)\Vert_{\textnormal{End}(\mathbb{C}^{d_\xi})}\leq C_{\alpha,\beta}\langle \xi\rangle^{m-\rho|\alpha|+\delta|\beta|},\, (x,\xi)\in G\times \widehat{G}.
\end{equation}The weight $\langle \xi \rangle:=(1+\lambda_{[\xi]})^\frac{1}{2}$ is defined in terms of the spectrum $\{\lambda_{[\xi]}\}_{[\xi]\in \widehat{G}}$ of the positive Laplacian $\mathcal{L}_G.$ Observe that the unitary dual of $G$ is a discrete set and the difference operators $\mathbb{D}^\alpha$ in \eqref{RTHormanderclasses} play the role of ``derivatives'' acting on functions/distributions defined on the unitary dual $\widehat{G}$.  

An important feature of the description above for the H\"ormander classes of pseudo-differential operators $\Psi^m_{\rho,\delta}(G),$ $m\in \mathbb{R},$ $0\leq \delta<\rho\leq 1,$ and $\rho\geq 1-\delta$ is that still, when $\rho<1-\delta,$ one can define the classes 
\begin{equation}\label{RT:xclasses}
    \Psi^m_{\rho,\delta}(G\times \widehat{G}):=\{A:C^\infty(G)\rightarrow C^\infty(G):\sigma_A \textnormal{ satisfies }\eqref{RTHormanderclasses}\}
\end{equation}allowing a well-defined class of pseudo-differential operators in the complete range $0\leq \delta<\rho\leq 1.$ The classes in  \eqref{RT:xclasses} become effective in handling certain classes of operators, for example resolvent operators
for vector fields on a compact Lie group $G$ which belong to the class $ \Psi^0_{0,0}(G\times \widehat{G})$ or in parametrices of H\"ormander sub-Laplacians that have symbols in the class  $ \Psi^{-1}_{\frac{1}{2},0}(G\times \widehat{G}).$
See  \cite{RuzhanskyTurunenWirth2014} for this and for other examples of the appearance of different symbol classes as
parametrices for hypoelliptic operators which cannot be handled by the standard theory
in view of the restriction in \eqref{rho:delta:restr}. 

The following Donnelly-Fefferman type inequality for elliptic pseudo-differential operators on a compact Lie group  is the main theorem of this work.

\begin{theorem}\label{Main:theorem} Let $0\leq \delta<\rho\leq 1.$ Let $A\in \Psi^m_{\rho,\delta}(G\times \widehat{G})$ be a positive elliptic pseudo-differential operator of order $m>0.$ Assume that $\sigma_A(x,\xi)\geq 0$ for all $(x,[\xi])\in G\times \widehat{G}.$  Let $(e_j,\lambda_j^m),$ $\lambda_j\geq 0,$ be the corresponding spectral data of $A,$ determined by the eigenvalue problem $Ae_j=\lambda_j^me_j$ with the eigenfunctions $e_j$ being $L^2$-normalised.   Then the following spectral estimates are valid:
\begin{itemize}
    \item For any non-empty open subset $\omega\subset G,$ we have
\begin{equation}\label{Spectral:Inequality:Intro}
    \Vert \varkappa\Vert_{L^2(G)}\leq C_1e^{C_2 {\lambda}}\Vert \varkappa\Vert_{L^2(\omega)},\,\,\,\varkappa\in \textnormal{span}\{e_j:\lambda_j\leq \lambda\},
\end{equation}with $C_1=C_1(\omega)$ and $C_2=C_2(\omega)$ depending on $\omega,$ but not on $\varkappa.$
\item For any $R>0$ let $B(x,R)$ be a ball defined by the geodesic distance, of radius $R>0$ and centred at $x.$ Then,
\begin{equation}\label{Donnelly-Fefferman}
    \sup_{B(x,2R)}|\varkappa|\leq  e^{C_1' {\lambda}+C_2'} \sup_{B(x,R)}|\varkappa|,\,\,\,\varkappa\in \textnormal{span}\{e_j:\lambda_j\leq \lambda\},
\end{equation}with $C_1'=C_{1}'(R)$ and $C_2'=C_2'(R)$  depending  only on the radius $R>0$ but  not on $\varkappa.$
\end{itemize}

\end{theorem}
\begin{remark}
In Subsection \ref{Applications:contro:theory} we give an application of this result to the control theory. More precisely, we use Theorem \ref{Main:theorem} to prove that the  heat equation
\begin{equation}\label{Heat:intro}
\begin{cases}u_t(x,t)+ A^\gamma u(x,t)=g(x,t)\cdot 1_\omega (x) ,& (x,t)\in G\times (0,T),
\\u(0,x)=u_0,\end{cases}
\end{equation}
associated to the fractional diffusion operator $A^{\gamma}$ is null-controllable at any time $T>0$ provided that $\gamma>1/m.$ The condition  $\gamma>1/m$ is sharp if one considers the case of the powers $A=\mathcal{L}_{\mathbb{T}}^{m/2}$ on the torus $G=\mathbb{T},$  see Miller \cite{Miller2006}.  For the terminology and for the basic aspects related to the control theory we refer the reader to Subsection \ref{Control:basics} and for the null-controllability result for the model \eqref{Heat:intro} see Theorem \ref{Main:theorem:statement} of Subsection \ref{Applications:contro:theory}. 
\end{remark}

\begin{remark}The analysis of growth estimates for eigenfunctions of the Laplacian and of other elliptic differential operators is still a problem of wide interest. In particular for its relation with the geometric analysis of nodal sets. For classic references on the subject we refer the reader to  Sunada \cite{Sunada}, Atiyah, Donnelly and Singer \cite{ADS82},  Borel and Garland \cite{BG83}, Jerison and Lebeau \cite{JerisonLabeau},  Donnelly and Fefferman \cite{DonnellyFefferman83,DonnellyFefferman,DonnellyFefferman1990,DonnellyFefferman1992}, Donnelly and Garofalo \cite{DonnellyGarofalo1992,DonnellyGarofalo1997}, and  Lin \cite{Lin1991}. As for recent works on the subject we refer the reader to  Apraiz, Escauriaza,  Wang,  and  Zhang \cite{Apraiz}, Blair and Sogge \cite{BlairSogge}, Cavalletti and  Farinelli \cite{CavalletiFarineli}, Enciso and Peralta-Salas \cite{EncisoPeralta}, Georgiev \cite{Georgiev},  Kenig, Zhu, and  Zhuge \cite{KZZ2022}, Logunov \cite{Logunov2018,Logunov20182},  Logunov, Malinnikova,  Nadirashvili, and  Nazarov \cite{Logunov2021},  Tian and  Yang \cite{TY2022} and Toth and Zelditch \cite{TothZelditch2021} just to mention a few. About the applications of spectral inequalities to the control theory we refer to  Benabdallah and  Naso \cite{BenabdallahNaso2002}, Fu,  L\"u, and  Zhang  \cite{FuLuZhang2020}, Lebeau and Robbiano \cite{LabeauRobbiano1995}, Lebeau and Zuazua \cite{LebeauLebeau1998}, J.-L. Lions, \cite{JLLions}, Micu and Zuazua \cite{MicuZuazua2006}, Miller \cite{Miller2006,Miller2007}, Rousseau and Lebeau \cite{RousseauLebeau2012}, Cardona \cite{Cardona2022},  Rousseau and Robbiano \cite{RousseauRobbiano2020},   and the extensive list of references therein.
\end{remark}

\subsection{Structure of the work}
In Section \ref{Preliminaries} we survey the rather extensive analytical
backgrounds about the theory of pseudo-differential operators on compact Lie groups with the calculus based on the matrix-valued quantisation and on compact manifolds (with the notion of a symbol in the $(\rho,\delta)$-class defined by local coordinate systems) in Section \ref{Preliminaries}. We do a particular  emphasis that these two points of view agree when $\rho\geq 1-\delta$ and when $0\leq \delta<\rho\leq 1.$ We then  use the global  theory of pseudo-differential operators and the matrix-valued quantisation to prove in Section \ref{SectionofProofs} the Donnelly-Fefferman/Lebeau-Robbiano spectral inequalities in Theorem  \ref{Main:theorem}. Finally, our application to the control theory of diffusion problems on compact Lie groups is addressed in Theorem \ref{Main:theorem:statement} of Subsection \ref{Applications:contro:theory}. 

\section{Preliminaries}\label{Preliminaries}
In this section, we present the preliminaries about the theory of pseudo-differential operators on compact Lie groups as well as the matrix-valued quantisation. For our further applications, we recall some results  about the control theory of heat equations on Hilbert spaces. The following standard notation will be employed during this work.
\begin{itemize}
    \item For two Hilbert spaces  $\mathcal{H}_1$ and $\mathcal{H}_2,$ we denote by $\mathscr{B}(\mathcal{H}_1,\mathcal{H}_2)$ the family of bounded and linear operators $T:\mathcal{H}_1\rightarrow \mathcal{H}_2.$
   
    \item The spectrum of a densely defined linear operator $A:\textnormal{Dom}(A)\subset{\mathcal{H}_1}\rightarrow \mathcal{H}_1$ will be denoted by $\sigma(A)$ and its resolvent set by $\textnormal{Resolv}(A):=\mathbb{C}\setminus \sigma(A). $ 
    \item We write $A\lesssim B$ if $A\leq cB$ where $c>0$ does not depend on $A$ and $B$. If $A\lesssim B$ and $B\lesssim A$ we write $A\asymp B.$
    \item $G$ is a compact Lie group and $\mathcal{L}_G$ denotes its corresponding Laplace-Beltrami operator.
\end{itemize}
\subsection{Pseudo-differential operators on compact Lie groups}

To define pseudo-differential operators, the main tool is the Fourier transform. On compact Lie groups the Fourier transform is defined in terms of the representations of a group. Only irreducible and unitary representations are needed to have the Fourier inversion formula. We define these objects as follows.
\subsubsection{The Fourier analysis of a compact Lie group}  
 The $L^p$-spaces $L^p(G)=L^p(G,dx)$ will be  associated with the Haar measure $dx.$ The Hilbert space $L^2(G)$ will be endowed with
   the inner product $(f,g)=\int\limits_{G}f(x)\overline{g(x)}dx.$  We will see that the spectral decomposition of $L^2(G)$ can be done in terms of the entries of unitary representations on a compact Lie group $G$.

A continuous and unitary representation of  $G$ on $\mathbb{C}^{\ell}$ is any continuous mapping $\xi\in\textnormal{Hom}(G,\textnormal{U}(\ell)) ,$ where $\textnormal{U}(\ell)$ is the Lie group of unitary matrices of order $\ell\times \ell.$ The integer number $\ell=\dim_{\xi}$ is called the dimension of the representation $\xi$ since it is the dimension of the representation space $\mathbb{C}^{\ell}.$

A subspace $W\subset \mathbb{C}^{d_\xi}$ is called $\xi$-invariant if for any $x\in G,$ $\xi(x)(W)\subset W,$ where $\xi(x)(W):=\{\xi(x)v:v\in W\}.$ The representation $\xi$ is irreducible if its only invariant subspaces are $W=\emptyset$ and $W=\mathbb{C}^{d_\xi},$ the trivial ones. On the other hand, 
any unitary representation $\xi$ is a direct sum of unitary irreducible representations. We denote it by $\xi=\xi_1\otimes \cdots\otimes \xi_j,$ with $\xi_i$ being irreducible representations on factors $\mathbb{C}^{d_{\xi_i}}$ that decompose the representation space $\mathbb{C}^{d_{\xi}}=\mathbb{C}^{d_{\xi_1}}\otimes \cdots \otimes \mathbb{C}^{d_{\xi_j}} .$

Two unitary representations $\xi\in \textnormal{Hom}(G,\textnormal{U}(d_\xi)) $ and $\eta\in \textnormal{Hom}(G,\textnormal{U}(d_\eta))$ are equivalent if there exists a linear mapping $F:\mathbb{C}^{d_\xi}\rightarrow \mathbb{C}^{d_\eta}$ such that for any $x\in G,$ $F\xi(x)=\eta(x)F.$ The mapping $F$ is called an intertwining operator between $\xi$ and $\eta.$ The set of all the intertwining operators between $\xi$ and $\eta$ is denoted by $\textnormal{Hom}(\xi,\eta).$ In view of the 1905's Schur lemma, if $\xi\in \textnormal{Hom}(G,\textnormal{U}(d_\xi)) $ is irreducible, then $\textnormal{Hom}(\xi,\xi)=\mathbb{C}I_{d_\xi}$ is formed by scalar multiples of the identity matrix  $I_{d_\xi}$ of order $d_\xi.$

The relation $\sim$ on the set of unitary representations $\textnormal{Rep}(G)$ defined by: {\it $\xi\sim \eta$ if and only if $\xi$ and $\eta$ are equivalent representations,} is an equivalence relation. The quotient 
$$
    \widehat{G}:={\textnormal{Rep}(G)}/{\sim}
$$is called the unitary dual of $G.$ It encodes all the Fourier analysis on the group. Indeed, if $\xi\in \textnormal{Rep}(G),$ the Fourier transform $\mathscr{F}_{G}$ associates to any $f\in C^\infty(G)$ a matrix-valued function $\mathscr{F}_{G}f$ defined on $\textnormal{Rep}(G)$ as follows
$$ (\mathscr{F}_{G}f)(\xi) \equiv   \widehat{f}(\xi)=\int\limits_Gf(x)\xi(x)^{*}dx,\,\,\xi\in \textnormal{Rep}(G). $$ The discrete Schwartz space $\mathscr{S}(\widehat{G}):=\mathscr{F}_{G}(C^\infty(G))$ is the image of the Fourier transform on the class of smooth functions. This operator admits a unitary extension from $L^2(G)$ into $\ell^2(\widehat{G}),$ with 
\begin{equation*}
 \ell^2(\widehat{G})=\{\phi:\forall [\xi]\in \widehat{G},\,\phi(\xi)\in \mathbb{C}^{d_\xi\times d_\xi}\textnormal{ and }\Vert \phi\Vert_{\ell^2(\widehat{G})}:=\left(\sum_{[\xi]\in \widehat{G}}d_{\xi}\Vert\phi(\xi)\Vert_{\textnormal{HS}}^2\right)^{\frac{1}{2}}<\infty \}.   
\end{equation*} The norm $\Vert\phi(\xi)\Vert_{\textnormal{HS}}$ is the standard Hilbert-Schmidt norm of matrices. The Fourier inversion formula takes the form
\begin{equation}
    f(x)=\sum_{[\xi]\in \widehat{G}}d_{\xi}\textnormal{Tr}[\xi(x)\widehat{f}(\xi)],\,f\in L^1(G),
\end{equation}where the summation is understood in the sense that from any equivalence class $[\xi]$ we choose one (any) a unitary representation. The sum is independent of such choice.  

\subsubsection{The quantisation formula} Let  $A:C^\infty(G)\rightarrow C^\infty(G)$ be a continuous linear operator with respect to the standard Fr\'echet structure on $C^\infty(G).$ The Schwartz kernel theorem associates to $A$ a kernel $K_A\in (C^\infty(G),\mathscr{D}'(G))$ such that
$$   Af(x)=\int\limits_{G}K_{A}(x,y)f(y)dy,\,\,f\in C^\infty(G).$$ The distribution defined via $R_{A}(x,xy^{-1}):=K_A(x,y)$ that provides the convolution identity
$$   Af(x)=\int\limits_{G}R_{A}(x,xy^{-1})f(y)dy,\,\,f\in C^\infty(G),$$
is called the right-convolution kernel of $A.$ By the Schwartz kernel theorem, 
  one can associate a global symbol $\sigma_A:G\times \widehat{G}\rightarrow \cup_{\ell\in \mathbb{N}}\mathbb{C}^{\ell\times \ell}$ to $A.$ Indeed, in view of the identity $Af(x)=(f\ast R_{A}(x,\cdot))(x),$ and after taking the Fourier transform with respect to $x\in G,$ we get 
 $$ \widehat{Af}(\xi)= \widehat{R}_{A}(x,\xi)\widehat{f}(\xi). $$ Then, the Fourier inversion formula gives the following representation of the operator $A$ in terms of the Fourier transform,
 \begin{equation}\label{Quantisation:formula}
     Af(x)=\sum_{[\xi]\in \widehat{G}}d_\xi\textnormal{Tr}[\xi(x)\widehat{R}_{A}(x,\xi)\widehat{f}(\xi)],\,f\in C^\infty(G).
 \end{equation} In view of the identity \eqref{Quantisation:formula}, from any equivalence class $[\xi]\in \widehat{G},$ we can choose one and only one irreducible unitary representation $\xi_0\in [\xi],$ such that the matrix-valued function
 \begin{equation}
    \sigma_{A}(x,[\xi])\equiv \sigma_A(x,\xi_0):=\widehat{R}_{A}(x,\xi_0),\,(x,[\xi])\in G\times\widehat{G},
 \end{equation}such that
\begin{equation}\label{Quantisation:formula2}
     Af(x)=\sum_{[\xi]\in \widehat{G}}d_\xi\textnormal{Tr}[\xi_0(x)\sigma_{A}(x,[\xi])\widehat{f}(\xi_0)],\,f\in C^\infty(G).
 \end{equation}The representation in \eqref{Quantisation:formula2} is independent of the choice of the representation $\xi_0$ from any equivalent class $[\xi]\in \widehat{G}.$ This is a consequence of the Fourier inversion formula.  In the following quantisation theorem we observe that the distribution $\sigma_A$ in \eqref{Quantisation:formula2} is unique and can be written in terms of the operator $A,$ see Theorems  10.4.4 and 10.4.6 of \cite[Pages 552-553]{Ruz}.
\begin{theorem}\label{The:quantisation:thm}
    Let $A:C^\infty(G)\rightarrow C^\infty(G) $ be a continuous linear operator. The following statements are equivalent.
    \begin{itemize}
        \item The distribution $\sigma_A(x,[\xi]):G\times \widehat{G}\rightarrow \cup_{\ell\in \mathbb{N}}\mathbb{C}^{\ell\times \ell}$ satisfies the quantisation formula
        \begin{equation}\label{Quantisation:3}
            \forall f\in C^\infty(G),\,\forall x\in G,\,\, Af(x)=\sum_{[\xi]\in \widehat{G}}d_\xi\textnormal{Tr}[\xi(x)\sigma_{A}(x,[\xi])\widehat{f}(\xi)].
        \end{equation}
        \item $ 
            \forall (x,[\xi])\in G\times \widehat{G},\, \sigma_{A}(x,\xi)=\widehat{R}_A(x,\xi).
        $ \\
        \item $ 
          \forall (x,[\xi]),\, \sigma_A(x,\xi)=\xi(x)^{*}A\xi(x),$ where $ A\xi(x):=(A\xi_{ij}(x))_{i,j=1}^{d_\xi}.  
        $ 
    \end{itemize}
\end{theorem}
\begin{remark}
    In view of the quantisation formulae \eqref{Quantisation:formula2} and \eqref{Quantisation:3}, a symbol $\sigma_A$ can be considered as a mapping defined on $G\times \widehat{G}$ or as a mapping  defined on $$G\times \textnormal{Rep}(G)$$ by identifying all the values $\sigma_A(x,\xi)=\sigma_A(x,\xi')=\sigma(x,[\xi])$ when $\xi',\xi\in [\xi].$
\end{remark}
\begin{example}[The symbol of a Borel function of the Laplacian] Let $\mathbb{X}=\{X_1,\cdots,X_n\}$ be an orthonormal basis of the Lie algebra $\mathfrak{g}.$ The positive Laplacian on $G$ is the second order differential operator 
\begin{equation}
    \mathcal{L}_G=-\sum_{j=1}^nX_j^2.
\end{equation}The operator $ \mathcal{L}_G$ is independent of the choice of the orthonormal basis $\mathbb{X}$ of $\mathfrak{g}.$ The $L^2$-spectrum of $\mathcal{L}_G$ is a discrete set that can be enumerated in terms of the unitary dual $\widehat{G}$,
\begin{equation}
    \textnormal{Spect}(\mathcal{L}_G)=\{\lambda_{[\xi]}:[\xi]\in \widehat{G}\}.
\end{equation}For a Borel function $f:\mathbb{R}^+_0\rightarrow \mathbb{C},$ the right-convolution kernel $R_{f(\mathcal{L}_G)}$ of the operator $f(\mathcal{L}_G)$ (defined by the spectral calculus) is determined by the identity
\begin{equation}
    f(\mathcal{L}_G)\phi(x)=\phi\ast R_{f(\mathcal{L}_G)}(x),\,x\in G.
\end{equation}This kernel satisfies the identity
\begin{equation}
    \widehat{R}_{f(\mathcal{L}_G)}([\xi])=f(\lambda_{[\xi]})I_{d_\xi}.
\end{equation}Then the matrix-valued  symbol of $f(\mathcal{L}_G)$ can be determined e.g. using  Theorem \ref{The:quantisation:thm} as follows
\begin{equation}
    \sigma_{f(\mathcal{L}_G)}(x,\xi)=\widehat{R}_{f(\mathcal{L}_G)}([\xi]).
\end{equation}Since the operator $f(\mathcal{L}_G)$ is central the symbol  $\sigma_{f(\mathcal{L}_G)}(\xi)=\sigma_{f(\mathcal{L}_G)}(x,\xi)$ does not depend of the spatial variable $x\in G.$ Of particular interest for us will be the Japanese bracket function
\begin{equation}
    \langle t\rangle:=(1+|t|)^{\frac{1}{2}},\,t\in \mathbb{R}.
\end{equation}In particular the symbol of the operator $ \langle \mathcal{L}_G\rangle$ is given by
\begin{equation}\label{Japanne:bracket:G}
     \sigma_{\langle \mathcal{L}_G\rangle}([\xi]):= \langle \xi \rangle I_{d_\xi}, \,\,\,\langle \xi \rangle:=\langle \lambda_\xi \rangle.  
\end{equation}

\end{example}
\subsubsection{H\"ormander classes of pseudo-differential operators on compact Lie groups} In this section we denote for any linear mapping $T$ on $\mathbb{C}^\ell$ by $\Vert T\Vert_{\textnormal{op}}$ the standard operator norm
$$ \Vert T\Vert_{\textnormal{op}}= \Vert T\Vert_{\textnormal{End}(\mathbb{C}^\ell)}:=\sup_{v\neq 0}\|Tv\|_{e}/\|v\|_{e} , $$ where $\|\cdot\|_{e} $ is the Euclidean norm.

For introducing the H\"ormander classes on compact Lie groups we have to measure the growth of derivatives of symbols in the group variable, for this we use vector fields $X\in T(G).$ To derivate symbols with respect to the discrete variable $[\xi]\in \widehat{G}$ we use difference operators. Before introducing the H\"ormander classes on compact Lie groups we have to define these differential/difference operators. 

So, if $\{X_{1},\cdots, X_{j}\}$ is an arbitrary family of left-invariant vector fields, we will denote by
$$  X_{x}^{\alpha}:=X_{1,x}^{\alpha_1}\cdots X_{n,x}^{\alpha_n} $$
an arbitrary canonical differential operator of order $m=|\alpha|.$
Also, we have to take derivatives with respect to the ``discrete'' frequency variable $\xi\in \textnormal{Rep}(G).$ To do this, we will use the notion of difference operators. Indeed,  the frequency variable in the symbol $\sigma_A(x,[\xi])$ of a continuous and linear operator $A$ on $C^\infty(G)$ is discrete. This is since $\widehat{G}$ is a discrete space.

If $\xi_{1},\xi_2,\cdots, \xi_{k},$ are  fixed irreducible and unitary  representation of $G$, which do not necessarily belong to the same equivalence class, then each coefficient of the matrix
\begin{equation}
 \xi_{\ell}(g)-I_{d_{\xi_{\ell}}}=[\xi_{\ell}(g)_{ij}-\delta_{ij}]_{i,j=1}^{d_{\xi_\ell}},\, \quad g\in G, \,\,1\leq \ell\leq k,
\end{equation} 
that is each function 
$q^{\ell}_{ij}(g):=\xi_{\ell}(g)_{ij}-\delta_{ij}$, $ g\in G,$ defines a difference operator
\begin{equation}\label{Difference:op:rep}
    \mathbb{D}_{\xi_\ell,i,j}:=\mathscr{F}_G(\xi_{\ell}(g)_{ij}-\delta_{ij})\mathscr{F}^{-1}_G.
\end{equation}
We can fix $k\geq \mathrm{dim}(G)$ of these representations in such a way that the corresponding  family of difference operators is admissible, that is, 
\begin{equation*}
    \textnormal{rank}\{\nabla q^{\ell}_{i,j}(e):1\leqslant \ell\leqslant k \}=\textnormal{dim}(G).
\end{equation*}
To define higher order difference operators of this kind, let us fix a unitary irreducible representation $\xi_\ell$.
Since the representation is fixed we omit the index $\ell$ of the representations $\xi_\ell$ in the notation that will follow.
Then, for any given multi-index $\alpha\in \mathbb{N}_0^{d_{\xi_\ell}^2}$, with 
$|\alpha|=\sum_{i,j=1}^{d_{\xi_\ell}}\alpha_{i,j}$, we write
$$\mathbb{D}^{\alpha}:=\mathbb{D}_{1,1}^{\alpha_{11}}\cdots \mathbb{D}^{\alpha_{d_{\xi_\ell},d_{\xi_\ell}}}_{d_{\xi_\ell}d_{\xi_\ell}}
$$ 
for a difference operator of order $m=|\alpha|$. Now, we are ready for introducing the global H\"ormander classes on compact Lie groups.
\begin{definition}[Global $(\rho,\delta)$-H\"ormander classes in the whole range $0\leq \delta,\rho\leq 1$]
    We say that $\sigma \in {S}^{m}_{\rho,\delta}(G\times \widehat{G})$ if the following symbol inequalities 
\begin{equation}\label{HormanderSymbolMatrix}
   \Vert {X}^\beta_x \mathbb{D}^{\alpha} \sigma(x,\xi)\Vert_{\textnormal{op}}\leqslant C_{\alpha,\beta}
    \langle \xi \rangle^{m-\rho|\gamma|+\delta|\beta|},
\end{equation} are satisfied for all $\beta$ and  $\gamma $ multi-indices and for all $(x,[\xi])\in G\times \widehat{G},$ where $ \langle \xi \rangle$ denotes the Japanese bracket function at $\lambda_{\xi}$ defined in \eqref{Japanne:bracket:G}.
\end{definition}
The class $\Psi^m_{\rho,\delta}(G\times \widehat{G})\equiv\textnormal{Op}({S}^m_{\rho,\delta}(G\times \widehat{G}))$ is defined by those continuous and linear operators on $C^\infty(G)$ such that $\sigma_A\in {S}^m_{\rho,\delta}(G\times \widehat{G}).$ 

In the next theorem we describe some fundamental properties of the global H\"ormander  classes of pseudo-differential operators  (\cite{Ruz}).
\begin{theorem}\label{RTcalculus:Group} Let $\rho,\delta\in [0,1]$ be such that $0\leqslant \delta\leqslant \rho\leqslant 1,$ $\rho\neq 1.$  Then $\Psi^\infty_{\rho,\delta}(G):=\cup_{m\in \mathbb{R}} \Psi^m_{\rho,\delta}(G)$ is an algebra of operators stable under compositions and adjoints, that is:
\begin{itemize}
    \item [-] the mapping $A\mapsto A^{*}:\Psi^{m}_{\rho,\delta}(G)\rightarrow \Psi^{m}_{\rho,\delta}(G)$ is a continuous linear mapping between Fr\'echet spaces. 
\item [-] The mapping $(A_1,A_2)\mapsto A_1\circ A_2: \Psi^{m_1}_{\rho,\delta}(G)\times \Psi^{m_2}_{\rho,\delta}(G)\rightarrow \Psi^{m_1+m_2}_{\rho,\delta}(G)$ is a continuous bilinear mapping between Fr\'echet spaces.
\end{itemize}Moreover, any operator in the class   $ \Psi^{0}_{\rho,\delta}(G)$ admits a  bounded extension from $L^2(G)$ to  $L^2(G).$
\end{theorem} 
\begin{remark}
    The $L^2$-boundedness result in Theorem \ref{RTcalculus:Group} is the global  version of the Calder\'on-Vaillancourt theorem for compact Lie groups. Moreover, if $A\in \Psi^0_{\rho,\delta}(G)$ is such that $0\leq \delta\leq \rho\leq 1,$  $\rho\neq 1,$ then
    \begin{equation}
        \Vert A\Vert_{\mathscr{B}(L^2)}\lesssim \sup\{ C_{\alpha,\beta}:|\alpha|+|\beta|\leq \ell\},
    \end{equation}where
    $$ C_{\alpha,\beta}:=\sup_{(x,[\xi])\in G\times\widehat{G} }   \langle\xi\rangle^{\rho|\alpha|-\delta|\beta|} \|\partial_{x}^\beta\mathbb{D}^\alpha \sigma_A(x,\xi)\|_{\textnormal{op}}$$ and $\ell\in \mathbb{N}_0$ is large enough.
\end{remark}

\subsubsection{H\"ormander classes of pseudo-differential operators on compact manifolds}\label{S2.1}
 Next, we shall present the basics related to the classes of pseudo-differential operators on a compact manifold without  boundary (closed manifold) by using  charts, see H\"ormander \cite{Hormander1985III} and  M. Taylor \cite{Taylorbook1981}.
 
\begin{definition}[Symbol classes on open sets]
Let $U\subset \mathbb{R}^n$ be an open  subset. The {\it symbol} $a\in C^{\infty}(U\times \mathbb{R}^n, \mathbb{C})$ belongs to the H\"ormander class  $S^m_{\rho,\delta}(U\times \mathbb{R}^n),$ with $0\leqslant \rho,\delta\leqslant 1,$ and $m\in \mathbb{R},$ if for every compact subset $K\subset U$ and for all $\alpha,\beta\in \mathbb{N}_0^n$, the inequalities
\begin{equation}\label{seminorms}
  |\partial_{x}^\beta\partial_{\xi}^\alpha a(x,\xi)|\leqslant C_{\alpha,\beta,K}(1+|\xi|)^{m-\rho|\alpha|+\delta|\beta|},
\end{equation} hold true uniformly in $(x,\xi)\in K\times \mathbb{R}^n.$
\end{definition}
\begin{remark}[Pseudo-differential operators on Euclidean open subsets]
 A continuous linear operator $A:C^\infty_0(U) \rightarrow C^\infty(U)$ (with respect to the standard Fr\'echet structure of $C^\infty_0(U)$ and of $C^\infty(U),$ respectively) 
is a pseudo-differential operator of order $m,$ and  of  $(\rho,\delta)$-type, if there exists
a symbol $a=a(x,\xi)$ in the class $S^m_{\rho,\delta}(U\times \mathbb{R}^n)$ such that
\begin{equation*}
 \forall f\in C^\infty_0(U),\,\forall x\in \mathbb{R}^n,\,\,   Af(x)=\int\limits_{\mathbb{R}^n}e^{2\pi i x\cdot \xi}a(x,\xi)(\mathscr{F}_{\mathbb{R}^n}{f})(\xi)d\xi,
\end{equation*} where
$$
    (\mathscr{F}_{\mathbb{R}^n}{f})(\xi):=\int\limits_{\mathbb{R}^n}e^{-2\pi i x\cdot \xi}f(x)dx
$$ is the  Euclidean Fourier transform of $f$ at $\xi\in \mathbb{R}^n.$ We denote the family of pseudo-differential operators with symbols in the class $S^m_{\rho,\delta}(U\times \mathbb{R}^n)$ by $\Psi^m_{\rho,\delta}(U).$   
\end{remark}
Now, let us extend the definition of the H\"ormander classes on Euclidean topological subspaces  to closed manifolds as follows. \begin{remark}[Pseudo-differential operators on compact manifolds]
Given a closed manifold $M,$ a  continuous linear operator $A:C^\infty(M)\rightarrow C^\infty(M) $ is a pseudo-differential operator of order $m,$ and of $(\rho,\delta)$-type, when $$  \rho\geqslant   1-\delta \textnormal{  and  } 0\leq \delta<\rho\leq 1,$$   if for every  coordinate patch $\omega: M_{\omega}\subset M\rightarrow U_{\omega}\subset \mathbb{R}^n,$
and for every $\phi,\psi\in C^\infty_0(U_\omega),$ the operator
\begin{equation*}
    Tu:=\psi(\omega^{-1})^*A\omega^{*}(\phi u),\,\,u\in C^\infty(U_\omega),
\end{equation*} is a  pseudo-differential operator with symbol $a_T\in S^m_{\rho,\delta}(U_\omega\times \mathbb{R}^n).$ Here, $\omega^{*}$ and $(\omega^{-1})^*$ are the pullbacks associated to the mappings $\omega$ and $\omega^{-1},$ respectively. All the operators $A$ with this property determines the family $A\in \Psi^m_{\rho,\delta}(M).$ 
\end{remark} 
\begin{remark}[The principal symbol of a pseudo-differential operator]
The symbol defined by localisations of a pseudo-differential operator $A$ is unique as an element in the quotient $\Psi^m_{\rho,\delta}(G)/\Psi^m_{\rho,\delta}(G).$ We call to this class  the {\it principal symbol of} $A.$  We denote it by
\begin{equation}
    a_m(x,\xi),\,\, (x,\xi)\in T^{*}M.\end{equation}
    The symbol is a well-defined section of the co-tangent bundle $T^*M$ if and only if
 $$  \rho\geqslant   1-\delta \textnormal{  and  } 0\leq \delta<\rho\leq 1.$$
Then, the main feature of the principal symbol of a pseudo-differential operator is that it remains invariant under changes of coordinates. 
\end{remark}

In the next result, we summarise  some fundamental properties of the   calculus  of pseudo-differential operators as defined by    H\"ormander (\cite{Hormander1985III}).
\begin{theorem}\label{calculus} Let $0\leqslant \delta<\rho\leqslant 1,$ be such that $\rho\geq 1-\delta.$ Then $\Psi^\infty_{\rho,\delta}(M):=\cup_{m\in \mathbb{R}} \Psi^m_{\rho,\delta}(M)$ is an algebra of operators stable under compositions and adjoints, that is:
\begin{itemize}
    \item [-] the mapping $A\mapsto A^{*}:\Psi^{m}_{\rho,\delta}(M)\rightarrow \Psi^{m}_{\rho,\delta}(M)$ is a continuous linear mapping between Fr\'echet spaces. 
\item [-] The mapping $(A_1,A_2)\mapsto A_1\circ A_2: \Psi^{m_1}_{\rho,\delta}(M)\times \Psi^{m_2}_{\rho,\delta}(M)\rightarrow \Psi^{m_1+m_2}_{\rho,\delta}(M)$ is a continuous bilinear mapping between Fr\'echet spaces.
\end{itemize}Moreover, any operator in the class   $ \Psi^{0}_{\rho,\delta}(M)$ admits a  bounded extension from $L^2(M)$ to  $L^2(M).$
\end{theorem} 
\begin{remark}
    The $L^2$-continuity statement of Theorem \ref{calculus} is the microlocalised version of the celebrated Calder\'on-Vaillancourt theorem, see the classical reference \cite{CalderonVaillancourt71}. Also, if $A\in \Psi^0_{\rho,\delta}(\mathbb{R}^n)$ is such that $0\leq \delta\leq \rho\leq 1,$  $\rho\neq 1,$ then ones has the estimate from above for the $L^2$-operator norm of $A$
    \begin{equation}
        \Vert A\Vert_{\mathscr{B}(L^2)}\lesssim \sup_{|\alpha|+|\beta|\leq [\frac{n}{2}]+1} C_{\alpha,\beta},
    \end{equation}where
    $$ C_{\alpha,\beta}:=\sup_{(x,\xi)\in \mathbb{R}^{2n}}   (1+|\xi|)^{\rho|\alpha|-\delta|\beta|} |\partial_{x}^\beta\partial_{\xi}^\alpha a(x,\xi)|.$$
\end{remark}
\begin{remark}\label{remark:CVTh}
    If $U\subset M$ is an open subset and the  dimension of $M $ is $n,$ for all $f\in C^\infty_0(U),$  by  microlocalising   $A\in \Psi^0_{\rho,\delta}(M)$ when $0\leq \delta < \rho\leq 1,$ and $\rho\geq 1-\delta,$ one has
    \begin{equation}
        \Vert A f\Vert_{L^2}\lesssim_{U} \sup_{|\alpha|+|\beta|\leq [\frac{n}{2}]+1} C_{\alpha,\beta,U}\Vert f\Vert_{L^2(G)},
    \end{equation}where
    $$ C_{\alpha,\beta,U}:=\sup_{(x,\xi)\in U\times\mathbb{R}^{n}}   (1+|\xi|)^{\rho|\alpha|-\delta|\beta|} |\partial_{x}^\beta\partial_{\xi}^\alpha a(x,\xi)|,$$ by making the identification of $U$ with an open subset of $\mathbb{R}^n.$
\end{remark}
\begin{remark}[Elliptic pseudo-differential operators]
   A pseudo-differential operator $A\in \Psi^m_{\rho,\delta}(M)$ is  elliptic of order $m,$ if  in any local coordinate system $U$, there exists $R=R_U>0,$ such that the symbol $a=a_U$ of $A$ associated to $U$ satisfies uniformly on any compact subset $K\subset{U}$ the growth estimate
\begin{equation}\label{Ellipticity}
   C_1 (1+|\xi|)^m\leq   |a(x,\xi)|\leq C_2 (1+|\xi|)^m, \, |\xi|\geq R,
\end{equation}uniformly in $(x,\xi)\in K\times \mathbb{R}^n.$  One of the main aspects of the spectral theory of elliptic pseudo-differential  operators is that their spectra are  purely discrete sets  (\cite{Hormander1985III}). 
\end{remark}
\begin{remark}
    The global H\"ormander classes on compact Lie groups can be used to describe the H\"ormander classes defined by local coordinate systems. We present the corresponding statement as follows.
\end{remark}
 
\begin{theorem}[Equivalence of classes in the range $0\leq \delta<\rho\leq 1,$ $\rho\geq 1-\delta,$ \cite{Ruz,RuzhanskyTurunenWirth2014}] Let $A:C^{\infty}(G)\rightarrow\mathscr{D}'(G)$ be a continuous linear operator and let $0\leq \delta<\rho\leq 1,$ with $\rho\geq 1-\delta.$ Then, $A\in \Psi^m_{\rho,\delta}(G)$ if and only if $A\in \Psi^m_{\rho,\delta}(G\times \widehat{G}).$ 
\end{theorem}

\subsubsection{Complex powers of an elliptic pseudo-differential operator on a compact Lie group} 
By using the Dunford-Riesz functional calculus,  to any  sector 
$   \Lambda\subset\mathbb{ C}, $ of the complex plane 
we will associate a class class of elliptic operators $\Psi^m_{\rho,\delta}(G\times \widehat{G};\Lambda)$ as developed by the third author and J. Wirth in \cite{RuzhanskyWirth2014}. There, one extended for any $0\leq \delta<\rho\leq 1$ the global functional calculus on compact manifolds due to Shubin \cite{Shubin1987} under the restrictions $0\leq \delta<\rho\leq 1$ and $0\leq \delta<\rho\leq 1.$ 

In  practice, $\Lambda$ will be any angle with the vertex at some complex number $z_0\in \mathbb{C}$. 
Now we will introduce the definition of parameter-elliptic global symbols as in \cite{RuzhanskyWirth2014}.

\begin{definition}[Parameter-ellipticity with respect to a sector $\Lambda$]
   Let $$ a=a(x,[\theta]):G\times \widehat{G}\rightarrow \bigcup_{[\xi]\in \widehat{G}}\mathbb{C}^{d_\xi\times d_\xi}$$  be a matrix-valued symbol. For $m\in \mathbb{R}^+,$  we say that $a$ is parameter-elliptic with respect to $\Lambda,$  if the following conditions are satisfied:
\begin{itemize}
    \item $\forall\lambda\in \Lambda,\, a(x,[\theta])-\lambda I_{d_\theta}\in \textnormal{GL}(d_\theta,\mathbb{C})$ is an invertible matrix. 
    \item The symbol inequality
\begin{equation*}
  \|( a(x,[\theta])-\lambda I_{d_\theta})^{-1}\|_{\textnormal{op}}\leqslant C(1+\langle\theta\rangle+|\lambda|^{\frac{1}{m}})^{-m},
\end{equation*} holds  uniformly in $x\in G,$ for all  $[\theta]\in \widehat{G}$ and all $\lambda\in \Lambda.$ In the case where $\Lambda=\{0\}$ is the trivial singleton, we just will say that the symbol $a(x,[\theta])$ is elliptic.
\end{itemize} 
 
\end{definition}
The following shows that we can use the notion of parameter-ellipticity in the construction of parametrices for the resolvent of an operator. The theorems below were proved in \cite{RuzhanskyWirth2014} and their corollaries are their immediate consequences. The proof of Lemma \ref{LemmaFC} below can be found in \cite[Section 7.2]{CardonaRuzhanskyC}.

\begin{theorem}\label{parameterparametrix}
Let $m>0,$ and let $0\leqslant \delta<\rho\leqslant 1.$ Let  $a(x,[\theta])$ be a parameter elliptic symbol with respect to a sector $\Lambda.$ Then  there exists a parameter-dependent parametrix of the resolvent $A-\lambda I,$ with matrix-valued symbol $a^{-\#}(x,\theta,\lambda)$ satisfying the estimates
\begin{equation*}
   \sup_{\lambda\in \Lambda}\sup_{(x,[\theta])\in G\times \widehat{G}}\Vert (|\lambda|^{\frac{1}{m}}+\langle\theta\rangle)^{m(k+1)}\langle\theta\rangle^{\rho|\alpha|-\delta|\beta|}\partial_{\lambda}^k{X}_x^{\beta}\mathbb{D}^{\alpha}a^{-\#}(x,\theta,\lambda)\Vert_{\textnormal{op}}<\infty,
\end{equation*}for all $\alpha,\beta\in \mathbb{N}_0^n$ and $k\in \mathbb{N}_0.$
\end{theorem} As a consequence of the previous theorem, we have an efficient classification of the symbol of the resolvent of an operator in the global H\"ormander classes, see \cite{RuzhanskyWirth2014}.
\begin{corollary}\label{resolv}
Let $m>0,$ and let $a\in S^{m}_{\rho,\delta}(G\times \widehat{G}) $ where  $0\leqslant \delta<\rho\leqslant 1.$ Let us assume that $\Lambda$ is a subset of the $L^2$-resolvent set of $A,$ $\textnormal{Resolv}(A):=\mathbb{C}\setminus \textnormal{Spec}(A).$ Then $A-\lambda I$ is invertible on $\mathscr{D}'(G)$ and the symbol of the resolvent operator $\mathcal{R}_{\lambda}:=(A-\lambda I)^{-1},$ $\widehat{\mathcal{R}}_{\lambda}(x,\xi)$ belongs to $S^{-m}_{\rho,\delta}(G\times \widehat{G}).$ 
\end{corollary}

\begin{remark}[Functional calculus of global pseudo-differential operators]\label{Global:complex:calculus}
    Let $a\in S^{m}_{\rho,\delta}(G\times \widehat{G})$ be a parameter elliptic symbol  of order $m>0$ with respect to the sector $\Lambda\subset\mathbb{C}.$ For $A=\textnormal{Op}(a),$ we will  define the operator $F(A)$  by the (Dunford-Riesz) complex functional calculus
\begin{equation}\label{F(A)}
    F(A)=-\frac{1}{2\pi i}\oint\limits_{\partial \Lambda_\varepsilon}F(z)(A-zI)^{-1}dz,
\end{equation}where
\begin{itemize}
    \item[(A1):] $\Lambda_{\varepsilon}:=\Lambda\cup \{z:|z|\leqslant \varepsilon\},$ $\varepsilon>0,$ and $\Gamma=\partial \Lambda_\varepsilon\subset\textnormal{Resolv}(A)$ is a positively oriented path in the complex plane $\mathbb{C}$.
    \item[(A2):] $F$ is an holomorphic function in $\mathbb{C}\setminus \Lambda_{\varepsilon},$ and continuous on its closure. 
    \item[(A3):] We will assume  decay of $F$ along $\partial \Lambda_\varepsilon$ in order that the operator \eqref{F(A)} will be densely defined on $C^\infty(G)$ in the strong sense of the topology on $L^2(G).$
\end{itemize}
\end{remark}

 Now, we will compute the matrix-valued symbols for operators defined by this complex functional calculus.
\begin{lemma}\label{LemmaFC}
Let $a\in S^{m}_{\rho,\delta}(G\times \widehat{G})$ be a parameter elliptic symbol  of order $m>0$ with respect to the sector $\Lambda\subset\mathbb{C}.$ Let $F(A):C^\infty(G)\rightarrow \mathscr{D}'(G)$ be the operator defined by the analytical functional calculus as in \eqref{F(A)}. Under the assumptions $\textnormal{(A1)}$, $\textnormal{(A2)}$, and $\textnormal{(A3)}$ of Remark \ref{Global:complex:calculus}, the matrix-valued symbol of $F(A),$ $\sigma_{F(A)}(x,\xi)$ is given by,
\begin{equation*}
    \sigma_{F(A)}(x,\xi)=-\frac{1}{2\pi i}\oint\limits_{\partial \Lambda_\varepsilon}F(z)\widehat{\mathcal{R}}_z(x,\xi)dz,
\end{equation*}where $\mathcal{R}_z=(A-zI)^{-1}$ denotes the resolvent of $A,$ and $\widehat{\mathcal{R}}_z\in S^{-m}_{\rho,\delta}(G\times \widehat{G}) $ is its symbol.
\end{lemma}

The decay assumption on $F$ will be clarified in the following theorem saying that the global calculus of pseudo-differential operators is stable under the action of the global complex functional calculus.
\begin{theorem}\label{DunforRiesz}
Let $m>0,$ and let $0\leqslant \delta<\rho\leqslant 1.$ Let  $a\in S^{m}_{\rho,\delta}(G\times \widehat{G})$ be a parameter elliptic symbol with respect to $\Lambda.$ Let us assume that $F$ satisfies the  estimate $|F(\lambda)|\leqslant C|\lambda|^s$ uniformly in $\lambda,$ for some $s<0.$  Then  the symbol of $F(A),$  $\sigma_{F(A)}\in S^{ms}_{\rho,\delta}(G\times \widehat{G}) $ admits an asymptotic expansion of the form
\begin{equation}\label{asymcomplex}
    \sigma_{F(A)}(x,\xi)\sim 
     \sum_{N=0}^\infty\sigma_{{B}_{N}}(x,\xi),\,\,\,(x,[\xi])\in G\times \widehat{G},
\end{equation}where $\sigma_{{B}_{N}}\in {S}^{ms-(\rho-\delta)N}_{\rho,\delta}(G\times \widehat{G})$ and 
\begin{equation*}
    \sigma_{{B}_{0}}(x,\xi)=-\frac{1}{2\pi i}\oint\limits_{\partial \Lambda_\varepsilon}F(z)(a(x,\xi)-z)^{-1}dz\in {S}^{ms}_{\rho,\delta}(G\times \widehat{G}).
\end{equation*}Moreover, 
\begin{equation*}
     \sigma_{F(A)}(x,\xi)\equiv -\frac{1}{2\pi i}\oint\limits_{\partial \Lambda_\varepsilon}F(z)a^{-\#}(x,\xi,\lambda)dz \textnormal{  mod  } {S}^{-\infty}(G\times \widehat{G}),
\end{equation*}where $a^{-\#}(x,\xi,\lambda)$ is the symbol of the parametrix to $A-\lambda I,$   in Corollary \ref{parameterparametrix}.
\end{theorem}

Now, we present the construction of the complex powers $A^z,$ $z\in \mathbb{C}.$
\begin{corollary}[Complex powers of elliptic operators on compact Lie groups]\label{Complex:Powers:Th}
Let $\varepsilon_>0.$ Let $A\in \Psi^m_{\rho,\delta}(G\times \widehat{G})$ be a parameter-elliptic pseudo-differential operator of order $m>0$ with respect to a  sector $\Lambda$ of the complex plane $\mathbb{C}.$ Let $\Lambda_{\varepsilon}:=\Lambda\cup \{z:|z|\leqslant \varepsilon\},$ $\varepsilon>0,$ and assume that $\Gamma=\partial \Lambda_\varepsilon\subset\textnormal{Resolv}(A)$ is a positively oriented curve in the complex plane $\mathbb{C}$. Then, the mapping
\begin{equation}\label{Contour:Integral}
    z\in \mathbb{C}\mapsto A^z:=-\frac{1}{2\pi i}\int\limits_{\partial \Lambda_\varepsilon}\lambda^z (A-\lambda I)^{-1}d\lambda\in \Psi^{\textnormal{Re}(z)m}_{\rho,\delta}(G).
\end{equation}
\end{corollary}
An immediate consequence of Theorem \ref{Complex:Powers:Th} is the construction of inverses for positive pseudo-differential  operators. We record it in the following way.
\begin{corollary}[Inverse of positive pseudo-differential operators]\label{The:Inverse:Theorem} Let Let $A\in \Psi^m_{\rho,\delta}(G\times\widehat{G})$ be a positive elliptic pseudo-differential operator of order  $m>0.$ Define $A^z$ via the contour integral \eqref{Contour:Integral} where $\Lambda$ is an acute angle  centred at the origin $0,$ with its interior containing the interval $(-\infty,0].$ Let $E_0:=\textnormal{Ker}(A)$ and let $E_0'$ be its orthogonal complement in $L^2(G).$ Then,
\begin{itemize}
    \item $A^z (E_0)=\{0\},$ and $A^z(E_0')\subset E_0'.$ 
    \item For any $z,w\in \mathbb{C},$  $A^{z+w}=A^zA^w.$ 
    \item If $P_0: L^2(G)\rightarrow E_0$ is the orthogonal projection on $E_0,$ then $A^0=I-P_0,$ and $A^{-1}$ is the inverse of the operator $A$ restricted to $E_0'.$  
\end{itemize}

\end{corollary}

\subsection{Global and local classes of pseudo-differential operators on the torus}\label{periodicclasses}
Let us consider the torus $\mathbb{T}^n\cong \mathbb{R}^n/\mathbb{Z}^n,$  $\mathbb{T}\cong \mathbb{S}^1.$
Different from the case of an arbitrary compact Lie group, here the local and the global H\"ormander classes agree for all $(\rho,\delta)\in [0,1]^2$ such that $0\leq \delta \leq 1,$ $0< \rho\leq 1.$ We will present the required preliminaries in order to give the statement of this equivalence. 

We will use the standard notation for this family of periodic pseudo-differential operators taken from \cite{Ruz}.
\begin{definition}[Discrete Schwartz space] The Schwartz space $\mathcal{S}(\mathbb{Z}^n)$ on the lattice $\mathbb{Z}^n$ is defined by the discrete functions $\phi:\mathbb{Z}^n\rightarrow \mathbb{C}$ verifying the inequality
 \begin{equation}
 \forall M\in\mathbb{R}, \exists C_M>0,\, |\phi(\xi)|\leq C_M(1+ |\xi|)^M.
 \end{equation}    
\end{definition}
\begin{definition}[The Fourier transform on $\mathbb{T}^n$]
  The toroidal  Fourier transform is defined for any test function $f\in C^{\infty}(\mathbb{T}^n)$ by
$$ \widehat{f}(\xi):=\int\limits_{\mathbb{T}^n}e^{-i2\pi\langle x,\xi\rangle}f(x)dx,\,\,\xi\in\mathbb{Z}^n.$$ Here, $dx$ stands for the normalised Haar measure on the torus.  
\end{definition}
\begin{remark}[The Fourier inversion formula]The Fourier inversion formula is given by the representation of any function $f\in L^1(\mathbb{T}^n)$ in its Fourier series $$ f(x)=\sum_{\xi\in \mathbb{Z}^n}e^{i2\pi\langle x,\xi \rangle }\widehat{f}(\xi),\,\, x\in\mathbb{T}^n. $$ 
    
\end{remark}
\begin{definition}[H\"ormander classes on the torus]
    The toroidal  H\"ormander class $S^m_{\rho,\delta}(\mathbb{T}^n\times \mathbb{R}^n), \,\, 0\leq \rho,\delta\leq 1,$ are defined by  those functions $a=a(x,\xi)$ which are smooth in $(x,\xi)\in \mathbb{T}^n\times \mathbb{R}^n$ and which satisfy the inequalities
\begin{equation}
|\partial^{\beta}_{x}\partial^{\alpha}_{\xi}a(x,\xi)|\leq C_{\alpha,\beta}(1+|\xi|)^{m-\rho|\alpha|+\delta|\beta|}.
\end{equation}
\end{definition}
\begin{remark}
    Note that symbols in $S^m_{\rho,\delta}(\mathbb{T}^n\times \mathbb{R}^n)$ are symbols in $S^m_{\rho,\delta}(\mathbb{R}^n\times \mathbb{R}^n)$ (see \cite{Ruz}) of order $m$ which are 1-periodic in $x.$ Then, if $a\in S^{m}_{\rho,\delta}(\mathbb{T}^n\times \mathbb{R}^n),$ the corresponding pseudo-differential operator is defined by the quantisation formula
\begin{equation}
a(X,D_{x})f(x)=\int\limits_{\mathbb{T}^n}\int\limits_{\mathbb{R}^n}e^{i2\pi\langle x-y,\xi \rangle}a(x,\xi)f(y)d\xi dy.
\end{equation}
\end{remark}
\begin{definition}[H\"ormander classes on the torus II]
The class $S^m_{\rho,\delta}(\mathbb{T}^n\times \mathbb{Z}^n),\, 0\leq \rho,\delta\leq 1,$ consists of  those functions $a(x, \xi)$ which are smooth in $x\in \mathbb{T}^n,$  for all $\xi\in\mathbb{Z}^n$ and which satisfy the symbol inequalities

\begin{equation}
\forall \alpha,\beta\in\mathbb{N}^n,\exists C_{\alpha,\beta}>0,\,\, |\Delta^{\alpha}_{\xi}\partial^{\beta}_{x}a(x,\xi)|\leq C_{\alpha,\beta}(1+|\xi|)^{m-\rho|\alpha|+\delta|\beta|}.
 \end{equation}

The operator $\Delta$ is the standard difference operator defined in $\mathbb{Z}^n,$ \cite{Ruz}. In this case for any $\alpha\in \mathbb{N}_0,$ $ \Delta^\alpha=\mathbb{D}^\alpha.$  The toroidal operator with symbol $a$ is defined as
\begin{equation}
a(x,D)f(x)=\sum_{\xi\in\mathbb{Z}^n}e^{i 2\pi\langle x,\xi\rangle}a(x,\xi)\widehat{f}(\xi),\,\, f\in C^{\infty}(\mathbb{T}^n).
\end{equation}
\end{definition}

\begin{remark}We denote the corresponding toroidal  class of operators associated with toroidal symbols in $S^m_{\rho,\delta}(\mathbb{T}^n\times \mathbb{Z}^n)$ (resp. $S^m_{\rho,\delta}(\mathbb{T}^n\times \mathbb{R}^n)$) by $\Psi^m_{\rho,\delta}(\mathbb{T}^n\times \mathbb{Z}^n),$ (resp. $\Psi^m_{\rho,\delta}(\mathbb{T}^n\times \mathbb{R}^n)$).
\end{remark}
There exists a process allowing the interpolation of the second argument of the symbols defined on $\mathbb{T}^n\times \mathbb{Z}^n$ in a smooth way to get a smooth symbol defined on $\mathbb{T}^n\times \mathbb{R}^n.$ It leads to the following toroidal equivalence-of-classes-theorem.  
\begin{theorem}\label{eq}
Let $(\rho,\delta)\in [0,1]^2$ be such that $0\leq \delta \leq 1,$ $0< \rho\leq 1.$ Then the symbol $a\in S^m_{\rho,\delta}(\mathbb{T}^n\times \mathbb{Z}^n)$ if only if there exists  an Euclidean symbol $a'\in S^m_{\rho,\delta}(\mathbb{T}^n\times \mathbb{R}^n)$ such that $a=a'|_{\mathbb{T}^n\times \mathbb{Z}^n}.$ Moreover, we have 
\begin{equation}\label{equivalence:of:classes}
    \Psi^m_{\rho,\delta}(\mathbb{T}^n\times \mathbb{Z}^n)=\Psi^m_{\rho,\delta}(\mathbb{T}^n\times \mathbb{R}^n) .
\end{equation}
Moreover, any $A \in\Psi^0_{\rho,\delta}(\mathbb{T}^n\times \mathbb{Z}^n) $ is bounded on  $L^2(\mathbb{T}^n),$ and 
\begin{equation}\label{Delgado:Ruzhansky:lp}
    \Vert A\Vert_{\mathscr{B}(L^2)}\lesssim \sup_{|\alpha|+|\beta|\leq [n/2]+1}  \sup_{(x,\xi)}\langle \xi \rangle^{\rho|\alpha|-\delta|\beta|}|\Delta^{\alpha}_{\xi}\partial^{\beta}_{x}a(x,\xi)|.
\end{equation}
\end{theorem}
\begin{proof} The proof of  \eqref{equivalence:of:classes} can be found in \cite{Ruz}. The proof of the $L^2$-estimate in \eqref{Delgado:Ruzhansky:lp} can be found in \cite{DRLp}. 
\end{proof}

\subsection{Null-controllability of diffusion problems on Hilbert spaces}\label{Control:basics}
This section is dedicated to presenting the functional analysis related to the control theory of fractional problems for self-adjoint linear operators on  Hilbert spaces, we will follow Miller \cite{Miller2006}.  We use the following notation/fact: 
\begin{itemize}
    \item  the norm of a Hilbert space $H$ will be denoted by $\Vert\cdot \Vert$ without using  subscript. In general $H_{1},$ $H_2,$ etc. denote Hilbert spaces. In what follows, any Hilbert space will be identified with its topological dual in the canonical way. 
\end{itemize}
\begin{remark}[Observation operator and Control operator]
   Let $H$ be a separable Hilbert space and let $A:\textnormal{Dom}(A)\subset H\rightarrow H$ be a positive self-adjoint operator with dense domain $\textnormal{Dom}(A)\subset{H}.$  Consider $H_1$ the Hilbert space obtained by choosing on the domain  $\textnormal{Dom}(A)$ the graph norm. We extend $\{e^{-tA}:t>0\}$ to a semigroup on the dual space $H_{1}^{*}.$ Let $S$ be an observation operator from  $H$ to a Hilbert space of inputs $U,$ and let us consider  the control operator $B=S^{*}\in \mathscr{B}(U, H_1^*).$ 
\end{remark}
Assume the following properties on $S$ and $B.$  
   \begin{assumption}\label{Assumption}  Let $S\in \mathscr{B}(H,U)$ be an observation operator and let us consider its adjoint (the control operator)  $B=S^{*}\in \mathscr{B}(U, H_1^*).$  Assume that, for some $T>0$ (and hence, for any $T>0$), the following estimates hold.
\begin{itemize}
    \item There exists $K_T>0,$ such that
    \begin{equation}\label{Hip:1}
        \forall v_0\in \textnormal{Dom}(A),\, \, \int\limits_0^T\Vert Se^{-tA}v_0\Vert^2\leq K_T\Vert v_0 \Vert^2.
    \end{equation}
    \item We have
    \begin{equation}\label{Hip:2}
       \forall u\in L^2_\textnormal{loc}(\mathbb{R},U),\,\Vert \int\limits_0^Te^{-tA} Bu(t)\Vert^2dt\leq K_T\int\limits_0^T\Vert u(t)\Vert^2dt. 
    \end{equation}
\end{itemize}
\end{assumption}
The assumptions  \eqref{Hip:1} and \eqref{Hip:2}identify the necessary hypotheses for the existence and  uniqueness of the solution of the model
\begin{equation}\label{Model:Hilbert}
    \phi_t+A\phi=Bu,\phi(0)=\phi_0\in {H},\, u\in L^2_{\textnormal{loc}}(\mathbb{R},U).
\end{equation}
We summarise this in the following result.
\begin{proposition} Under the hypothesis \eqref{Hip:1} and \eqref{Hip:2}, for any input $u\in L^2_{\textnormal{loc}}(\mathbb{R},U),$ there exists a unique solution $u\in C(\mathbb{R}^+_0,U)$ to \eqref{Model:Hilbert} such that
\begin{equation}\label{Solution:Hilbert}
    \phi(t)=e^{-tA}\phi_0+\int\limits_0^T e^{(s-t)}Bu(s)ds.
\end{equation}
\end{proposition}
We precise the notion of null-controllability in the following definition.
\begin{definition}\label{defi}The model \eqref{Model:Hilbert} is null-controllable in time $T>0,$ if for any initial state $\phi_0\in H,$ there exists an input function $u\in L^2_{\textnormal{loc}}(\mathbb{R}^+_0,H)$ such that its solution \eqref{Solution:Hilbert} satisfies $\phi(T)=0.$ 
\end{definition}
\begin{remark} Let us consider the adjoint model  to  \eqref{Model:Hilbert}  without the source term, that is
\begin{equation}
    v_t+Av=0.
\end{equation}
Since $B=S^{*}$ the null-controllability of \eqref{Model:Hilbert}  is equivalent to the following observability inequality: there exists $C_T>0,$ such that 
\begin{equation}
    \forall v_0\in H, \, \Vert e^{-TA}v_0\Vert\leq C_T \Vert S e^{-tA}v_0\Vert_{L^2((0,T),U)}.
\end{equation}
The smallest constant $C_T>0$ is called the cost of controllability in time $T>0$. Note that by the duality argument, the cost of controllability in time $T>0,$ is the smallest constant $C_T>0$ satisfying that
\begin{equation}
    \forall \phi_0\in H,\exists u \textnormal{ in Definition \eqref{defi} such that } \Vert u \Vert_{L^2((0,T),U)}\leq C_T\Vert \phi_0\Vert.
\end{equation}
\end{remark}    
\begin{remark}[Spectral inequalities and  null-controllability]
     Theorem \ref{Miller:Theorem} says that a spectral inequality for the power $A^\gamma$ which is defined by the functional calculus of the operator $A,$ is a sufficient condition for the null-controllability of the 
model \eqref{Model:Hilbert}.
Next, we give the precise statement. Here, $E_\lambda^{T}:=E^{T}(0,\lambda):H\rightarrow H$ denotes an arbitrary projection of the spectral measure $\{E_\lambda^{T}\}_{\lambda>0} $ associated to a positive and densely defined linear operator $T:H\rightarrow H.$
\end{remark}

\begin{theorem}[Miller \cite{Miller2006}, 2006]\label{Miller:Theorem} Assume that for some $\gamma\in (0,1),$ the fractional operator $A^\gamma$ satisfies the spectral inequality
\begin{equation}\label{Espectral:Inequality:Hilbert}
    \forall\lambda>0,\,\forall u\in E_\lambda^{A^{\gamma}} (H),\exists d_1,d_2>0,\,\,\, \Vert  v\Vert\leq d_1e^{d_2 \lambda}\Vert  Sv\Vert.
\end{equation}Then, the problem \eqref{Model:Hilbert} is null-controllable in time $T>0.$ Moreover,  the controllability cost $C_T$ over short times $T,$ satisfies the inequality
\begin{equation}
    \forall \beta>\frac{\gamma}{1-\gamma},\,\exists C_1,C_2,\, \forall T\in (0,1),\, C_T\leq C_1 e^{C_2T^{-\beta}}.
\end{equation}
\end{theorem}

\section{Donnelly-Fefferman inequalities on compact Lie groups}\label{SectionofProofs}

In this section we prove our main Theorem \ref{Main:theorem}. 
We employ the following notation. 
\begin{itemize}
    \item[-] We denote by $\omega\neq \emptyset$  an open  non-empty subset in $G,$ and  a generic compact subset in $M$ will be denoted by $K.$  
    
    \item[-] For any $T>0,$ let us consider the space-time manifold $$G_T:=G\times [0,T],$$ and the Sobolev space $H^s(G_T)$ of order $s\in \mathbb{N},$ is defined by the norm 
      \begin{equation}  
      \Vert f\Vert_{H^s(G_T)}^2=\sum_{0\leq j\leq s}\int\limits_0^T\int\limits_G\left[\left|\partial_t^jf(x,t)\right|^2+|(1+\mathcal{L}_G)^{\frac{j}{2}}f(x,t)|^2\right]dxdt<\infty,
    \end{equation}where $\mathcal{L}_G$ is the positive Laplace-Beltrami operator on $G.$     
\end{itemize}

\begin{remark}[A topological construction]\label{A topological construction}For our further analysis we will make a topological construction. We do it by the following steps.
    \begin{itemize}
        \item  Step 1. Let us consider the cylinder $G_{T}=G\times [0,T].$ See Figure \ref{Fig1}.
        \item Step 2.  We fix a parameter $\varepsilon>0$ and we extend the cylinder $G_{T}$ in the time-variable $t$ until obtaining a new cylinder $G_{T,\varepsilon}=G\times [-T-\varepsilon,T+\varepsilon],$ and we identify its lateral boundaries $G\times\{-T-\varepsilon\}\sim G\times\{T+\varepsilon\}. $ See Figure \ref{Fig2}.
        \item Step 3. After the identification $G\times\{-T-\varepsilon\}\sim G\times\{T+\varepsilon\} $ the lateral boundaries of the  manifold $G_{T,\varepsilon}=G\times [-T-\varepsilon,T+\varepsilon]$ can be glued until obtaining the Lie group ${G}\times \mathbb{T}({T,\varepsilon})$ where $\mathbb{T}({T,\varepsilon})$ is the flat torus  $$ \mathbb{T}({T,\varepsilon}):=\mathbb{R}/2(T+\varepsilon)\mathbb{Z}\cong [-(T+\varepsilon),T+\varepsilon].$$ The manifold ${G}\times \mathbb{T}({T,\varepsilon})$ can be seen as a ``torus'' where any transversal section is a copy of  $G.$ See Figure \ref{Fig3}.
    \end{itemize}        
    \end{remark}
\begin{remark}[The operator $-\partial_t^2+A^{\frac{2}{  m  }}$ on ${G}\times \mathbb{T}({T,\varepsilon})$]  Let $A\in \Psi^m_{\rho,\delta}(G\times \widehat{G})$ be a positive elliptic matrix-valued pseudo-differential operator of order $m>0.$ 
    The operator $-\partial_t^2$ became, up to a constant, the (positive) Laplace-Beltrami type operator on $\mathbb{T}({T,\varepsilon}).$ 
    
    Note that $-\partial_t^2\in \Psi^2_{1,0}(\mathbb{T}({T,\varepsilon}))$ is a positive and elliptic differential operator of second order on $\mathbb{T}({T,\varepsilon}).$ Indeed, let us consider the orthogonal basis of $L^2(\mathbb{T}({T,\varepsilon}))$ formed by the exponential functions
\begin{equation}
 t\in \mathbb{T}({T,\varepsilon})\mapsto   \tilde{e}^\varepsilon_{k}(t)=\exp\left(\frac{2\pi i t k}{2(T+\varepsilon)}\right),\,\, k\in \mathbb{Z},
\end{equation}and let us consider the $L^2$-normalised system of $2(T+\varepsilon)$-periodic eigenfunctions $${e}^\varepsilon_{k}:=\tilde{e}^\varepsilon_{k}/\sqrt{2(T+\varepsilon)},$$ of the Laplacian $-\partial_t^2.$ The global symbol  of $-\partial_t^2$ is given by  
$$\sigma_{-\partial_t^2}(k)=\frac{4\pi^2 k^2}{4(T+\varepsilon)^2}=\left(\frac{\pi k}{T+\varepsilon}\right)^2,\,k\in\mathbb{Z}.   $$
The ellipticity of $-\partial_t^2$ follows from the following inequality
\begin{equation}
   \exists C_{1},C_{2}>0,\, \forall k\in \mathbb{Z},\,\, C_{1}|k|^2\leq |\sigma_{-\partial_t^2}(k)|\leq C_{2}|k|^2.
\end{equation}   Note that the constants $C_{1}$ and $C_{2}$ are independent of $\varepsilon>0$ if  $\varepsilon\in (0,1).$ 
Note that   $$\mathcal{A}(x,t,D,\partial_t)=-\partial_t^2+A^{\frac{2}{  m  }} \in \Psi^2_{1,\delta}({G}\times \mathbb{T}({T,\varepsilon}))$$ is also a positive and elliptic pseudo-differential operator on the Lie group ${G}\times \mathbb{T}({T,\varepsilon}).$ Note that $G_T$ can be viewed as  an open sub-manifold of the Lie group $$  {G}\times \mathbb{T}({T,\varepsilon})\cong G\times\mathbb{T}.$$ 

\end{remark}

\begin{figure}[h]
\includegraphics[width=7cm]{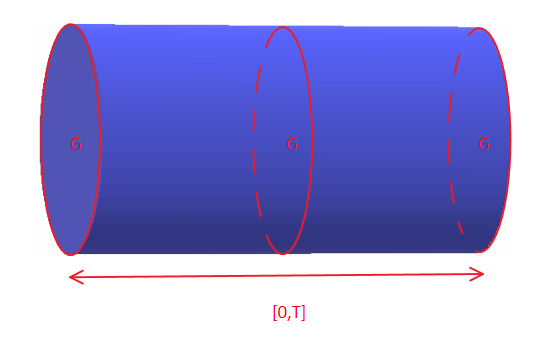}\\
\caption{Step 1: To consider the space-time manifold $G_{T}=G\times [0,T].$}
 \label{Fig1}
\centering
\end{figure}
\begin{figure}[h]
\includegraphics[width=9.5cm]{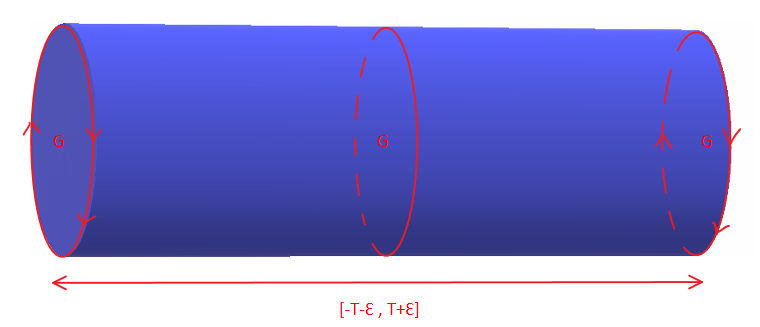}\\
\caption{We extend the cylinder $G_{T}$ in the time-variable $t$ until obtaining a new cylinder $G_{T,\varepsilon}=G\times [-T-\varepsilon,T+\varepsilon],$ and we identify its lateral boundaries $G\times\{-T-\varepsilon\}\sim G\times\{T+\varepsilon\}. $}
 \label{Fig2}
\centering
\end{figure}
\begin{figure}[h]
\includegraphics[width=6.5cm]{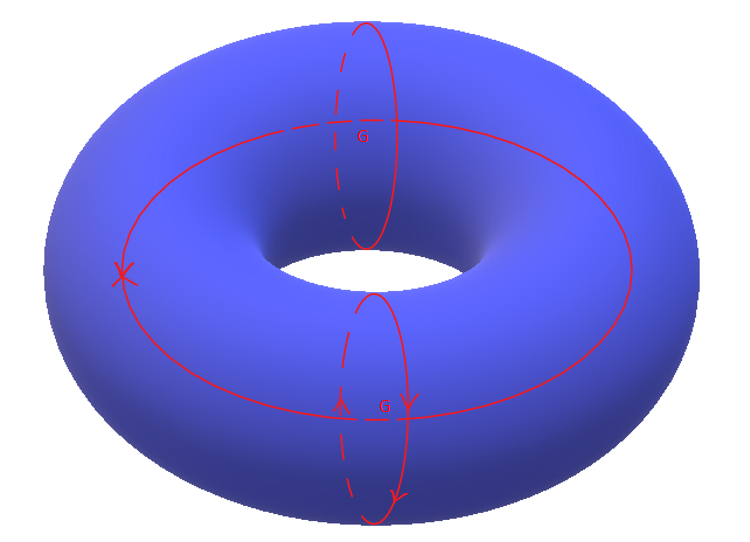}\\
\caption{We have constructed the Lie group ${G}\times \mathbb{T}({T,\varepsilon})$ where $\mathbb{T}({T,\varepsilon})$ is the flat torus  $ \mathbb{T}({T,\varepsilon}):=\mathbb{R}/2(T+\varepsilon)\mathbb{Z}\cong [-(T+\varepsilon),T+\varepsilon].$ This closed manifold is a compact Lie group. }
 \label{Fig3}
\centering
\end{figure}

Our analysis starts with the   spectral inequality in
Proposition \ref{Lemma:LR;Ineq} below that corresponds to the spectral  inequality \eqref{Spectral:Inequality:Intro} in Theorem \ref{Main:theorem} in the case where $A$ satisfies the lower bound $A\geq cI,$ for some $c>0.$ Then, the proof of \eqref{Spectral:Inequality:Intro} in Theorem \ref{Main:theorem} will be deduced from this particular situation.

\begin{proposition}\label{Lemma:LR;Ineq}  Let $A\in \Psi^m_{\rho,\delta}(G\times \widehat{G})$ be a positive elliptic pseudo-differential operator of order $m>0.$ Assume that $\sigma_A(x,\xi)\geq 0$ for all $(x,[\xi])\in G\times \widehat{G}.$ Moreover, assume that for some $c>0,$  $A\geq cI$ in $L^2(G),$ and that $\sigma_A(x,[\xi])\geq cI_{d_\xi}$ on every representation space.

 Then, for any non-empty open subset $\omega\subset M,$   any $a_j\in \mathbb{R},$ and all $\lambda>0,$ the following spectral inequality holds 
\begin{equation}\label{ObservabilityInequality}
    \left(\sum_{\lambda_j\leq \lambda}a_j^2\right)^\frac{1}{2}\leq C_1e^{C_2{\lambda}}\left\Vert \sum_{\lambda_j\leq \lambda}a_je_j(x)  \right\Vert_{L^2(\omega)},\end{equation}
    where $C_{1}>0$ and $C_{2}>0$ may depend on $\omega$ but not on $a_j,$ $\lambda>0$ or on the eigenfunctions $e_j.$
\end{proposition}
We postpone the proof of Proposition \ref{Lemma:LR;Ineq} for a moment. Indeed,   for our further analysis we require the following interpolation inequality. It is formulated in the case of the compact Lie groups but it is still valid in the case of a compact Riemannian manifold $(M,g),$ see e.g. \cite{LebeauLebeau1998}.
\begin{lemma}\label{CarlemanInequality} Let us consider the operator $L(x,t,D,\partial_t)=-\partial_t^2+\mathcal{L}_G\in \Psi^2_{1,0}(G_T)$ with  $G_T=G\times (0,T)$. Let $\omega$ be a non-empty open subset  in $M.$ 

Then, for any $T>0$ and all $\alpha\in (0,T/2),$ there exists $\delta\in (0,1)$ such that
\begin{equation}\label{Global:Carleman:Estimate}
    \Vert \phi \Vert_{H^1(G\times (\alpha,T-\alpha))}\leq C\Vert \phi\Vert_{H^1(G_T)}^\delta\left(\Vert L(x,t,D,\partial_t)\phi \Vert_{L^2(G\times (0,T))} +\Vert \partial_t\phi\Vert_{L^2(\omega)} \right)^{1-\delta},
\end{equation}for all $\phi\in H^2(G\times (0,T))$ such that $\phi=0$ in $G\times \{0\}.$
\end{lemma}  We refer the reader to Rousseau and Lebeau \cite{RousseauLebeau2012} for extensions of this result even, for second-order elliptic operators with Lipchitz coefficients. In the following section we describe the approach that we introduce for the proof of Proposition \ref{Lemma:LR;Ineq}.
\subsection{Sketch of the proof of Proposition \ref{Lemma:LR;Ineq}}
As it was proved by Jerison and Lebeau \cite{JerisonLabeau}, the interpolation inequality \eqref{Global:Carleman:Estimate} can be used to prove the inequality \eqref{ObservabilityInequality} for the Laplacian $\mathcal{L}_G$ (or even for the Laplacian $-\Delta_g$ on an arbitrary compact Riemannian manifold $(M,g)$). From the point of view of the pseudo-differential calculus the operators 
    $$  L(x,t,D,\partial_t)=-\partial_t^2+\mathcal{L}_G \textnormal{   and   }\mathcal{A}(x,t,D,\partial_t)=-\partial_t^2+A^{\frac{2}{  m  }}$$  are similar. They are elliptic pseudo-differential operators of order 2. So, in order to prove \eqref{ObservabilityInequality}, we will construct a suitable function $\phi\in H^{2}$ in terms of the eigenfunctions $e_j,$ were the index $j$ is such that $\lambda_j\leq \lambda,$ and from the inequality in \eqref{Global:Carleman:Estimate}, we will deduce an inequality of the form
    \begin{equation}\label{Global:Carleman:Estimate:U}
    \Vert \phi \Vert_{H^1(G\times (\alpha,T-\alpha))}\leq C\Vert \phi\Vert_{H^1(G_T)}^\delta\left(\Vert \mathcal{A}(x,t,D,\partial_t)\phi \Vert_{L^2(G\times (0,T))} +\Vert \partial_t\phi\Vert_{L^2(\omega)} \right)^{1-\delta},
\end{equation} and by following the strategy by Jerison and Lebeau in \cite{JerisonLabeau} from the inequality in  \eqref{Global:Carleman:Estimate:U} we will deduce the inequality \eqref{ObservabilityInequality}.

Now, observe that the interpolations inequalities \eqref{Global:Carleman:Estimate} and \eqref{Global:Carleman:Estimate:U} are essentially similar. Indeed, they differ just by the  $L^2$-norms
$$ \Vert L(x,t,D,\partial_t)\phi \Vert_{L^2(G\times (0,T))}\textnormal{   and   }\Vert \mathcal{A}(x,t,D,\partial_t)\phi \Vert_{L^2(G\times (0,T))}.$$ Informally, if we were able to compute the inverse of the operator $\mathcal{A}(x,t,D,\partial_t),$ then  observe that the composition 
$$ L(x,t,D,\partial_t)  \mathcal{A}(x,t,D,\partial_t)^{-1}$$ would be a pseudo-differential  operator of order zero. Then, an application of a suitable ``Calderon-Vaillancourt'' theorem would provide the $L^2$-boundedness of the operator $L(x,t,D,\partial_t)  \mathcal{A}(x,t,D,\partial_t)^{-1}.$ All this informal argument is constructed under the pseudo-differential philosophy, indeed, the $L^2$-boundedness of such an operator would give the following estimate 
$$
    \Vert L(x,t,D,\partial_t)\phi \Vert_{L^2}= \Vert L(x,t,D,\partial_t)  \mathcal{A}(x,t,D,\partial_t)^{-1} \mathcal{A}(x,t,D,\partial_t)\phi \Vert_{L^2}$$ 
    $$\lesssim \Vert \mathcal{A}(x,t,D,\partial_t)\phi \Vert_{L^2}.
$$
Then, if all this works, from  \eqref{Global:Carleman:Estimate} we could obtain \eqref{Global:Carleman:Estimate:U}. So, the main difficulty in computing the inverse of the operator $\mathcal{A}(x,t,D,\partial_t)=-\partial_t^2+A^{\frac{2}{  m  }}$ is that it is acting on distributions defined in the cylinder $G_{T}=G\times [0,T].$ This is a compact manifold whose lateral boundaries are  $G\times\{0\}$ and $ G\times\{T\}, $ and the global pseudo-differential calculus in \cite{Ruz} does not allow the construction of parametrices and inverses on this kind of manifold. However, inspired a little bit by the topological constructions by Donaldson in \cite{Donaldson}, we will fix a parameter $\varepsilon>0$ and we will embed the cylinder $$G_T=G\times [0,T] $$ into the manifold
$$  G\times \mathbb{T}({T,\varepsilon})$$ that was constructed in Remark \ref{A topological construction}, (see Figures \ref{Fig1}, \ref{Fig2} and \ref{Fig3}). Our strategy will be to construct the inverse/parametrix of the operator $\mathcal{A}(x,t,D,\partial_t)=-\partial_t^2+A^{\frac{2}{  m  }}$ on the compact Lie group $ G\times \mathbb{T}({T,\varepsilon}),$ and then by using the calculus in \cite{Ruz} and the parameter-ellipticity notion developed in \cite{RuzhanskyWirth2014} we will prove an inequality of the form \eqref{Global:Carleman:Estimate:U}. The dependence of the parameter $\varepsilon$ will be eliminated forcing it to go to zero and showing that the auxiliary inequality that we get is stable under this limit procedure. Moreover, the inequality that we obtain will be given for the manifold $G\times [-T,T]$ but a symmetry property in the auxiliary function $\phi$ (that we construct later) will give us the required interpolation inequality on the interval $[0,T].$ To do this, we force the function $\phi=\phi(x,t)$ to be odd with respect to the time variable, that is $\phi(x,t)=-\phi(x,-t).$ Summarising, the proof of Proposition \ref{Lemma:LR;Ineq} will be constructed by  following three steps:
\begin{itemize}
    \item Step 1: To compute  the inverse of the operator $\mathcal{A}(x,t,D,\partial_t).$ We do this in Subsection \ref{Inverse:Subsection}.
    \item Step 2: To establish the  $L^2$-theory for the operator $(-\partial_t^2+\mathcal{L}_G)\mathcal{A}(x,t,D,\partial_t)^{-1}.$ This will be done in Subsection \ref{L2:theory}.
    \item Step 3: To use the Steps 1 and 2 in the proof of Proposition \ref{Lemma:LR;Ineq}. This will be presented in Subsection \ref{Begininnig}.
\end{itemize}

\subsection{Computing the inverse of  $\mathcal{A}(x,t,D,\partial_t)$ }\label{Inverse:Subsection} According to the hypothesis in Proposition \ref{Lemma:LR;Ineq},
let us consider the operator $A\geq cI.$
 We will analyse  the  invertibility of the operator $$\mathcal{A}(x,t,D,\partial_t)=-\partial_t^2+A^{\frac{2}{  m  }}:H^2(G\times \mathbb{T}({T,\varepsilon}))\rightarrow L^2(G\times \mathbb{T}({T,\varepsilon})).$$ 

Note that we have embedded  the manifold $G_{T}$ (with lateral boundary $\partial G_{T}=(G\times \{0\})\cup (G\times \{T\})$) on the compact Lie group $G\times \mathbb{T}({T,\varepsilon}).$  
In the next lemma  we prove that $\mathcal{A}(x,t,D,\partial_t)$ is parameter elliptic with respect to a suitable angle of the complex plane. 

\begin{lemma}\label{parameterparametrix:A}Let us consider the operator $A\geq cI$ of Proposition \ref{Lemma:LR;Ineq}. Then, it is parameter-elliptic with respect to the sector
\begin{equation}\label{Left:sector}
    \Lambda=\{\lambda+i\lambda':|\lambda'|\leq -\lambda,\,\lambda\leq 0\}.
\end{equation}
    
\end{lemma}
\begin{proof}In order to prove that $A$ is parameter elliptic with respect to the sector in \eqref{Left:sector}, see Figure \ref{Sector},
\begin{figure}[h]
\includegraphics[width=6.5cm]{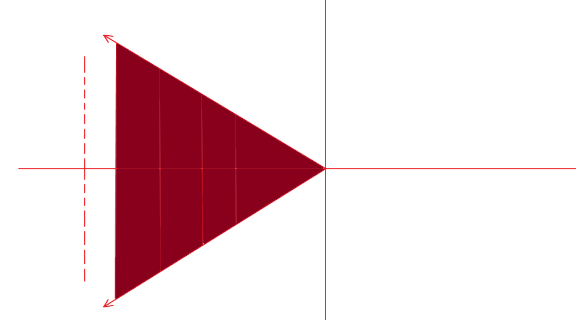}\\
\caption{The sector in \eqref{Left:sector}}
 \label{Sector}
\centering
\end{figure} we have to prove that the matrix-valued symbol $a(x,\theta)=\sigma_{A}(x,[\theta])$ of $A$ satisfies the following estimate \begin{equation*}
  \|( a(x,[\theta])-z I_{d_\xi})^{-1}\|_{\textnormal{op}}\leqslant C(1+\langle\theta\rangle+|z|^{\frac{1}{m}})^{-m},
\end{equation*}   uniformly in $x\in G,$ for all  $[\theta]\in \widehat{G}$ and all $z\in \Lambda.$ Note that $A$ is elliptic and then, we have the estimate
\begin{equation}\label{Ellip:coro}
    \Vert a(x,[\theta])\Vert_{\textnormal{op}}\geq C\langle \theta \rangle^{m},\,\,[\theta]\in \widehat{G}.
\end{equation}In some basis of the representation space, in view of the positivity of its symbol $a(x,[\theta])\geq 0,$ we can write it in diagonal form, that is
\begin{equation}
    a(x,[\theta])\equiv a(x,\theta)=\textnormal{diag}[\lambda_{11}(x,\theta),\cdots , \lambda_{d_\theta d_\theta}(x,\theta) ],\,\lambda_{ii}(x,\theta)\geq 0,\,1\leq i\leq d_\theta.
\end{equation}
Then, using that the order of $A$ is $m$ and \eqref{Ellip:coro} we have that
\begin{equation}\label{sup:estimate}
    \Vert a(x,[\theta])\Vert_{\textnormal{op}}=\sup_{1\leq i\leq d_\theta}  \lambda_{ii}(x,\theta)\asymp \langle\theta \rangle^m.
\end{equation}Note that the positivity hypothesis  $a=\sigma_A(x,[\theta])\geq cI_{d_\theta}$ in  Proposition \ref{Lemma:LR;Ineq}, implies the invertibility of the matrix-valued symbol $a=a(x,[\theta])$ in any representation space. Then, the the symbol

$$a(x,[\theta])^{-1}$$  is elliptic of order $-m.$ In particular it satisfies the inequality
$$ \forall [\theta]\in \widehat{G},\,\, \Vert a(x,[\theta])^{-1}\Vert_{\textnormal{op}}=\sup_{1\leq i\leq d_\theta}  \lambda_{ii}(x,\theta)^{-1}\asymp \langle\theta \rangle^{-m}.   $$ Then we have that
\begin{equation}\label{inf:estimate}
    \inf_{1\leq i\leq d_\theta}  \lambda_{ii}(x,\theta)\asymp \langle\theta \rangle^m.
\end{equation}Note that \eqref{sup:estimate} and \eqref{sup:estimate} imply that for any $z=\lambda+i\lambda'\in \Lambda$ we have
\begin{align*}
   \|( a(x,[\theta])-z I_{d_\xi})^{-1}\|_{\textnormal{op}} &\leqslant\sup_{1\leq i\leq d_\theta}  |\lambda_{ii}(x,\theta)-z|^{-1}=\sup_{1\leq i\leq d_\theta}  |\lambda_{ii}(x,\theta)-\lambda-i\lambda'|^{-1}\\
  &\lesssim \sup_{1\leq i\leq d_\theta} ( \lambda_{ii}(x,\theta)-\lambda+|\lambda'|)^{-1}\\
  &\lesssim( \langle \theta\rangle^m-\lambda+|\lambda'|)^{-1}\\
  &\asymp ( \langle \theta\rangle +(-\lambda+|\lambda'|)^{\frac{1}{m}})^{-m}\\
  &\asymp ( \langle \theta\rangle +|z|^{\frac{1}{m}})^{-m},
\end{align*}as desired. The proof of Lemma \ref{parameterparametrix:A} is complete.   
\end{proof}

\begin{remark}[Complex powers of $\mathcal{A}(x,t,D,\partial_t)$] Let us consider the sector in \eqref{Left:sector}, see Figure \ref{Sector}. Let $c>0$ be the lower bound in the condition on the operator $A\geq cI$ in  Proposition \ref{Lemma:LR;Ineq}.  If $$ 0<\varepsilon<\frac{c^{\frac{2}{m}}}{1000},$$ 
 consider the complex sector  
 $
     \Lambda_{\varepsilon}=\Lambda\cup\{z\in \mathbb{C}:|z|\leq \varepsilon\}
 $ in Figure \ref{Sector2} below.
\begin{figure}[h]
\includegraphics[width=6.5cm]{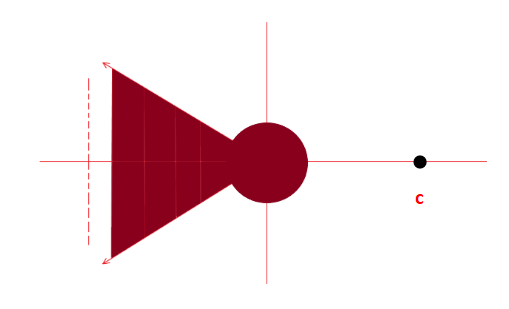}\\
\caption{The new sector $\Lambda_{\varepsilon}=\Lambda\cup\{z\in \mathbb{C}:|z|\leq \varepsilon\}.$}
 \label{Sector2}
\centering
\end{figure} 

Note that any $z\in \Lambda$ belongs to the resolvent of $\mathcal{A}(x,t,D,\partial_t).$ Indeed, the lower bound $A\geq cI,$ $c>0,$ implies that the spectrum of $A$  is contained in the infinite interval $[c,\infty).$ Moreover, the spectral mapping theorem implies that the spectrum of  $A^{\frac{2}{m}}$ is contained in $[c^{\frac{2}{m}},\infty).$ Since $-\partial_t^2$ is a positive operator on $L^2(\mathbb{T}({T,\varepsilon}))$ and $\lambda=0$ belongs to its spectrum, we have that
$$ \mathcal{A}(x,t,D,\partial_t)=-\partial_t^2+A^{\frac{2}{m}}  $$ is positive on $L^2(G\times \mathbb{T}({T,\varepsilon}))$ with its spectrum contained in $[c^{\frac{2}{m}},\infty).$ This analysis proves the inclusion $\Lambda_\varepsilon\subset \textnormal{Resolv}( \mathcal{A}(x,t,D,\partial_t)).$

In view of Lemma \ref{parameterparametrix:A} and of Theorem \ref{Complex:Powers:Th} we have that 
\begin{equation}
  z\in \mathbb{C}\mapsto  G^z:= \mathcal{A}(x,t,D,\partial_t)^{z},
\end{equation}is a holomorphic family of  pseudo-differential operators, that maps any $z\in \mathbb{C}$ into the class $\Psi^{2\textnormal{Re}(z)}_{\rho,\delta}(G\times \mathbb{T}({T,\varepsilon})),$ where
\begin{equation}\label{Dunfors-Riesz}
   G^zf(x)=-\frac{1}{2\pi i}\int\limits_{\partial\Lambda_\varepsilon }{\lambda^z}(\mathcal{A}(x,t,D,\partial_t)-\lambda I)^{-1}f(x)d\lambda,\,\, f\in C^\infty(G\times \mathbb{T}({T,\varepsilon})).
\end{equation}
In particular, with $z=-1,$ we have the inverse $G^{-1},$
\begin{equation}\label{Dunfors-Riesz:Inverse}
   G^{-1}f(x)=-\frac{1}{2\pi i}\int\limits_{\partial\Lambda_\varepsilon}{\lambda^{-1}}(\mathcal{A}(x,t,D,\partial_t)-\lambda I)^{-1}f(x)d\lambda,\,\, f\in C^\infty(G\times \mathbb{T}({T,\varepsilon})),
\end{equation} of $\mathcal{A}(x,t,D,\partial_t)$ on the orthogonal complement of its kernel. This is, if $P_0$ is the orthogonal projection on the subspace  $\textnormal{Ker}(\mathcal{A}(x,t,D,\partial_t)),$ $G^{-1}G=I-P_0,$ see Corollary \ref{The:Inverse:Theorem}. In view of the lower bound $A\geq cI,$ we deduce that $P_0$ is the null operator.
\end{remark}
  Moreover, we have the following property.
\begin{proposition}\label{The:constant:B}
Let $0<\varepsilon<1,$ and let us consider the operator norm 
$$ B_\varepsilon=\Vert \mathcal{A}(t,x,D,\partial_t)^{-1} \Vert_{\mathscr{B}(L^2(G\times \mathbb{T}({T,\varepsilon})),  H^2(G\times \mathbb{T}({T,\varepsilon})) )}.  $$
Then
\begin{equation}
    B:=\sup_{0<\varepsilon<1}B_\varepsilon \leq 1+1/c,
\end{equation}where $c>0$ in the constant is the positivity condition $A\geq cI$ of Proposition \ref{Lemma:LR;Ineq}. 
\end{proposition}
\begin{proof}Let us consider the orthogonal basis of $L^2(\mathbb{T}({T,\varepsilon}))$ formed by the exponential functions
\begin{equation}
 t\in \mathbb{T}({T,\varepsilon})
 \mapsto   \tilde{e}^\varepsilon_{k}(t)=\exp\left(\frac{2\pi i t k}{2(T+\varepsilon)}\right),\,\, k\in \mathbb{Z},
\end{equation}and let us consider the $L^2$-normalised system of $2(T+\varepsilon)$-periodic eigenfunctions $${e}^\varepsilon_{k}:=\tilde{e}^\varepsilon_{k}/\sqrt{2(T+\varepsilon)},$$ of the Laplacian $-\partial_t^2.$ The corresponding eigenvalues of $-\partial_t^2$ are given by  
$$ \mu_{k,\varepsilon}=\frac{4\pi^2 k^2}{4(T+\varepsilon)^2}=\left(\frac{\pi k}{T+\varepsilon}\right)^2,\,k\in\mathbb{Z}.   $$ Since $\{{e}^\varepsilon_{k}\otimes e_j\}$ is a basis for $L^2(G\times \mathbb{T}({T,\varepsilon}) )$ the spectrum of the operator $-\partial_t^2+A^{\frac{2}{  m  }}$ is determined by the sequence
$$   \mu_{k,\varepsilon}+\lambda_j^2= \left(\frac{\pi k}{T+\varepsilon}\right)^2+\lambda_j^2,\,\,k\in \mathbb{Z},\,j\in \mathbb{N}_0.$$
Since $A\geq cI,$  we have the eigenvalue inequality $\lambda_k\geq c,$ and then for any $f\in L^2(G\times \mathbb{T}({T,\varepsilon}) )$ we have that
\begin{align*}
    \Vert \mathcal{A}(t,x,D,\partial_t)^{-1} f\Vert_{H^2}^2 &=\left\Vert \sum_{k,j} ( \mu_{k,\varepsilon}+\lambda_j^2)^{-1}(f, {e}^\varepsilon_{k}\otimes e_j){e}^\varepsilon_{k}\otimes e_j\right\Vert_{
    H^2}^2\\
    & =\sum_{k,j} (1+ \mu_{k,\varepsilon}+\lambda_j^2)^2( \mu_{k,\varepsilon}+\lambda_j^2)^{-2}|(f, {e}^\varepsilon_{k}\otimes e_j)|^2\\
    & \leq \left(\frac{1}{c}+1\right)^2\sum_{k,j} |(f, {e}^\varepsilon_{k}\otimes e_j)|^2=\left(\frac{1}{c}+1\right)^2\Vert f \Vert_{L^2}^2.
\end{align*}From the previous analysis we deduce that $B\leq 1+\frac{1}{c}$ as desired.
\end{proof}

\subsection{$L^2$-theory for the operator $(-\partial_t^2+\mathcal{L}_G)\mathcal{A}(x,t,D,\partial_t)^{-1}$}\label{L2:theory}
The global Calder\'on-Vaillancourt  theorem is a sharp $L^2$-estimate for pseudo-differential operators. We will apply it in the proof of our spectral inequality. Indeed,  let us consider the operator $A\geq cI$ of Proposition \ref{Lemma:LR;Ineq}. Observe that
$$ (-\partial_t^2+\mathcal{L}_G)\in \Psi^2_{1,0}(G \times \mathbb{T}({T,\varepsilon}) \times \widehat{G}\times 2(T+\varepsilon)\mathbb{Z}),$$   belongs to the H\"ormander class of order 2. Also, we have that
$$ \mathcal{A}(x,t,D,\partial_t)=-\partial_t^2+A^{\frac{2}{m}}  \in \Psi^2_{1,\delta}(G \times \mathbb{T}({T,\varepsilon}) \times \widehat{G}\times 2(T+\varepsilon)\mathbb{Z}), $$
and for the inverse of $\mathcal{A}(x,t,D,\partial_t)$ we have  $$ \mathcal{A}(x,t,D,\partial_t)^{-1}\in \Psi^{-2}_{1,\delta}(G \times \mathbb{T}({T,\varepsilon}) \times \widehat{G}\times 2(T+\varepsilon)\mathbb{Z}).$$  The pseudo-differential calculus implies that
$$ F(x,t,D,\partial_t):=(-\partial_t^2+\mathcal{L}_G)\mathcal{A}(x,t,D,\partial_t)^{-1}\in \Psi^0_{1,\delta}(G \times \mathbb{T}({T,\varepsilon}) \times \widehat{G}\times 2(T+\varepsilon)\mathbb{Z}) .$$
The global Calder\'on-Vaillancourt theorem implies that $F(x,t,D,\partial_t)$ is a bounded operator on $L^2(G \times \mathbb{T}({T,\varepsilon}) \times \widehat{G}\times 2(T+\varepsilon)\mathbb{Z}).$ Note that the global quantisation allows writing the operator $F(x,t,D,\partial_t)$ as follows
\begin{equation}
    F(x,t,D,\partial_t)u(x,t)=\sum_{k\in \mathbb{Z}^n}\sum_{[\xi]\in \widehat{G}}d_\xi \textnormal{Tr}[(\xi\otimes e^{\frac{i2\pi (\cdot,k)}{2(T+\varepsilon)}})(x,t)\sigma(x,t,\xi,k)\widehat{u}(\xi,k)],
\end{equation}where, for any $u\in C^\infty_0(G\times \mathbb{T}({T,\varepsilon})),$  the Fourier transform of $u$
at 
$$ (\xi,k)\in   \widehat{G\times\mathbb{T} }({T,\varepsilon})\cong \widehat{G}\times  2(T+\varepsilon)\mathbb{Z}, $$
is defined by
\begin{equation}
    \widehat{u}(\xi,k)=\int\limits_{\mathbb{T}({T,\varepsilon})}\int\limits_{G}e^{-\frac{i2\pi (t,k)}{2(T+\varepsilon)}}\xi(x)^{*}{u}(x,t)dx dt,\,k\in \mathbb{Z},\,[\xi]\in \widehat{G}.
\end{equation}
For the toroidal variable, we have used the toroidal calculus, see Subsection \ref{periodicclasses} or  \cite{Ruz} for details. From now, $\Delta_k$ denotes the difference operator on a lattice.  
Note that the matrix-valued symbol $\sigma_F(x,t,\xi,k)$ of the operator $  F(x,t,D,\partial_t)$ admits an asymptotic expansion of the form
\begin{equation*}
     \sigma_F(x,t,\xi,k)\sim \sum_{j=0}^\infty \sigma_{m-j}(x,t,\xi,k),\,\,(x,t,\xi,k)\in G\times \mathbb{T}({T,\varepsilon})\times\widehat{G}\times  2(T+\varepsilon)\mathbb{Z},
\end{equation*}in the sense that
\begin{equation*}
 \forall N\in \mathbb{N},\,    \sigma_F(x,t,\xi,k)- \sum_{j=0}^N \sigma_{m-j}(x,t,\xi,k)\in S^{m-(N+1)(1-\delta)}(G\times\mathbb{T}({T,\varepsilon})\times\widehat{G}\times  2(T+\varepsilon)\mathbb{Z}).
\end{equation*}
Let $a(x,\xi)^{\frac{2}{  m  }}$ denote the matrix-valued symbol of $A^{\frac{2}{  m  }},$ (defined by the functional calculus of matrices).
The matrix-valued component $\sigma_{F}$ of higher order of the quotient operator $$ F(x,t,D,\partial_t):=(-\partial_t^2+\mathcal{L}_G)\mathcal{A}(x,t,D,\partial_t)^{-1}$$  is given by
\begin{equation}
     \sigma_F(x,t,\xi,k)=\sigma^{\varepsilon}_F(x,t,\xi,k):=\left({\left(\frac{\pi k}{T+\varepsilon}\right)^2+\langle\xi\rangle^2}\right)\left({\left(\frac{\pi k}{T+\varepsilon}\right)^2 I_{d_\xi}+a(x,[\xi])^{\frac{2}{m}}}\right)^{-1}.
\end{equation}
Since the ellipticity of $A^{\frac{2}{  m  }},$ implies that$$ C_1\langle\xi\rangle^2\leq\Vert a(x,\xi)^{\frac{2}{  m  }}\Vert_{\textnormal{op}} \leq C_2\langle\xi\rangle^2, \,\,[\xi]\in \widehat{G},  $$  the symbol $\sigma_F$ is elliptic of order zero, satisfying  the inequality
$$   \tilde{C}_1\leq \| \sigma_F(x,t,\xi,k)\|\leq \tilde{C}_2,$$ with $C_{1}$ and $C_2$ independent of $\varepsilon\in (0,1).$
In view of the positivity hypothesis  $\sigma_{A}(x,[\xi])\geq cI_{d_\xi}$ on every representation space, the family \begin{equation}
    [0,1]\mapsto \sigma_F(x,t,\xi,k)=\sigma^{\varepsilon}_F(x,t,\xi,k)
\end{equation} is a smooth function from the unit interval $[0,1]$ to the class $$ S^{0}_{1,\delta}(G\times\mathbb{T}({T,\varepsilon})\times\widehat{G}\times  2(T+\varepsilon)\mathbb{Z}),$$ endowed with its natural Fr\'echet structure.   As a consequence the supremum
$$   \sup_{\varepsilon\in [0,1]}\sup_{(x,t,[\xi],k)}(1+|k|+\langle \xi\rangle)^{|\alpha|+|\gamma|-\delta|\beta|} \|\partial_{x,t}^\beta\mathbb{D}^\alpha\Delta_k^{\gamma} \sigma_F^\varepsilon(x,t,\xi,k)\|_{\textnormal{op}}<\infty,$$
is bounded. This implies that $\sigma_F$ satisfies inequalities of the type
\begin{equation*}
  \|\partial_{x,t}^\beta\mathbb{D}^\alpha\Delta_k^{\gamma} \sigma_F(x,t,\xi,k)\|_{\textnormal{op}}  \leqslant C_{\alpha,\beta,\gamma}(1+|k|+\langle \xi\rangle)^{-|\alpha|-|\gamma|+\delta|\beta|},
\end{equation*} where the constants $ C_{\alpha,\beta,\gamma}$ are independent of the parameter $\varepsilon\in [0,1].$
Since the Calder\'on-Vaillancourt estimates the $L^2$-boundendess of $F$ in terms of the constants $C_{\alpha,\beta,\gamma}$ and of $\tilde{C}_2,$ (see Remark \ref{remark:CVTh}) that is, for any $u\in C^\infty_0(\mathbb{T}({T,\varepsilon})\times \tilde{U}),$
\begin{equation}\label{Estimate:Calderon:Vaillancourt}
    \Vert  F(x,t,D,\partial_t) u\Vert_{L^2}\leq \left(\sup_{|\alpha|+|\beta|+|\gamma|\leq  \ell }{ \{C_{\alpha,\beta,\gamma}, \tilde{C_2}\}  }\right)\Vert u\Vert_{L^2},
\end{equation} where $\ell\in \mathbb{N}$ is big enough. As a consequence of this discussion we have proved the following lemma.
\begin{lemma}\label{Finite:Constant:CV} Let $0<\varepsilon<1,$ and let us consider the operator norm 
$$ C_\varepsilon=\Vert (-\partial_t^2+\mathcal{L}_G)\mathcal{A}(t,x,D,\partial_t)^{-1} \Vert_{\mathscr{B}(L^2(G\times \mathbb{T}({T,\varepsilon})))}.  $$
Then
\begin{equation}
    C:=\sup_{0<\varepsilon<1}C_\varepsilon <\infty.
\end{equation}Moreover, there is $\ell\in \mathbb{N}_0$ large enough such that 
    \begin{equation}
     C\lesssim    \left(\sup_{|\alpha|+|\beta|+|\gamma|\leq  \ell }{ \{C_{\alpha,\beta,\gamma}, \tilde{C_2}\}  }\right)\Vert u\Vert_{L^2}.
    \end{equation}
\end{lemma}

\subsection{Proof of Proposition \ref{Lemma:LR;Ineq}}\label{Begininnig} We shall reduce the proof of this proposition to an interpolation inequality. We explain this strategy as follows.
\begin{remark}\label{About:the:proof:of:propo}Let us consider a spectra parameter $\lambda>0.$ Let $\varkappa
\in \textnormal{Im}(\textnormal{E}_{\lambda}).$ Then, $\varkappa$ can be written as linear combinations of the eigenfunctions $e_j,$ where $\lambda_j\leq \lambda,$ that is
\begin{equation}
    \varkappa(x)=\sum_{\lambda_j\leq \lambda}a_je_j(x).
\end{equation}
We note that for the proof of \eqref{ObservabilityInequality}, is enough to show that
\begin{equation}\label{The:main:inequality:here}
    F(x,t):=\sum_{\lambda_j\leq \lambda}\frac{\sinh(\lambda_jt)}{\lambda_j} a_je_j(x),\,\,(x,t)\in G_T:=G\times [0,T],
\end{equation}satisfies the interpolation inequality
\begin{equation}\label{Interpolation:Inequality}
    \Vert F \Vert_{H^1(G\times (\alpha,T-\alpha))}\leq C\Vert F\Vert_{H^1(G_{T})}^\kappa \Vert \varkappa \Vert_{L^2(\omega)}^{1-\kappa}.
\end{equation}
Indeed, by the  Parseval theorem  we have that
\begin{align*} 
\Vert F\Vert^2_{H^1(G\times (\alpha,T-\alpha))} &\geq C \Vert F\Vert^2_{L^2(G\times (\alpha,T-\alpha))}\\
&=\int\limits_\alpha^{T-\alpha}\int\limits_G\left|\sum_{\lambda_j\leq \lambda}\frac{\sinh(\lambda_jt)}{\lambda_j} a_je_j(x)\right|^2dx\,dt\\
&=\sum_{\lambda_j\leq \lambda}|a_j|^2\int\limits_\alpha^{T-\alpha}\left|\frac{\sinh(\lambda_jt)}{\lambda_j}\right|^2dt\\
&\geq \sum_{\lambda_j\leq \lambda}|a_j|^2\int\limits_\alpha^{T-\alpha}t^2dt\\
&=C_{\alpha}\sum_{\lambda_j\leq \lambda}|a_j|^2.
\end{align*}
Observing that
\begin{equation}
   \partial_tF(x,0)= \sum_{\lambda_j\leq \lambda}a_je_j(x),
\end{equation}
and  that
$$   \Vert F\Vert_{H^1(G_T)}^2\lesssim e^{2T\lambda}\lambda^2\sum_{\lambda_j\leq \lambda}|a_j|^2$$
we deduce the inequality
\begin{equation}
    C_{\alpha}\sum_{\lambda_j\leq \lambda}|a_j|^2\lesssim_{\alpha,T}\left( e^{2T\lambda}\lambda^2\sum_{\lambda_j\leq \lambda}|a_j|^2\right)^{\kappa}\left\Vert \sum_{\lambda_j\leq \lambda}a_je_j(x)  \right\Vert_{L^2(\omega)}^{2(1-\kappa)}.
\end{equation}Hence
\begin{align*}
    \left( \sum_{\lambda_j\leq \lambda}|a_j|^2\right)^{1-\kappa}\lesssim e^{2\varkappa T\lambda}\lambda^{2 \varkappa } \left\Vert \sum_{\lambda_j\leq \lambda}a_je_j(x)  \right\Vert_{L^2(\omega)}^{2(1-\kappa)},
\end{align*}and consequently
\begin{align*}
    \left( \sum_{\lambda_j\leq \lambda}|a_j|^2\right)^{\frac{1}{2}}\lesssim e^{{\varkappa  T\lambda/(1-\kappa)}}\lambda^{{\varkappa /(1-\kappa)} }\left\Vert \sum_{\lambda_j\leq \lambda}a_je_j(x)  \right\Vert_{L^2(\omega)},
\end{align*}proving \eqref{ObservabilityInequality}. Note also that we can estimate $e^{{ \varkappa T\lambda/(1-\kappa)}}\lambda^{{\varkappa /(1-\kappa)} }\lesssim e^{{C_{2}\varkappa  T\lambda/(1-\kappa)}}$ for some $C_{2}>0.$     
\end{remark}

\begin{proof}[Proof of  Proposition \ref{Lemma:LR;Ineq}] In view of Remark \ref{About:the:proof:of:propo}
we proceed with the proof of the inequality \eqref{Interpolation:Inequality}.
Note that, by normalising $\varkappa$ on $L^2(G)$ we can assume without loss of generality  that $\Vert \varkappa\Vert_{L^2(G)}=1.$
\subsubsection{ An auxiliar interpolation inequality on $[0,T+\varepsilon)$.}
Let $\varepsilon\in (0,1)$ be a positive parameter whose conditions will be imposed later.  Firstly, by replacing in  Lemma \ref{CarlemanInequality} the open interval $I_T:=(0,T)$ by $I_{T+\varepsilon}:=(0,T+\varepsilon),$ and with  $ \tilde{G}_{T+\varepsilon}:=G\times (0,T+\varepsilon), $ we shall make use of the following interpolation inequality:\\

\fbox{%
\parbox{\textwidth}{
 { \it 
  For any $T>0$ and all $\alpha\in (0,T/2),$ there exists $\kappa\in (0,1)$ such that
\begin{equation}\label{Global:Carleman:EstimateProofSpectralIneq} 
    \Vert \phi \Vert_{H^1(G\times (\alpha,T-\alpha))}\leq C\Vert \phi\Vert_{H^1(G_{T+\varepsilon})}^\kappa\left(\Vert (-\partial_t^2+\mathcal{L}_G)\phi \Vert_{L^2(G_{T+\varepsilon})} +\Vert \partial_t\phi(x,0)\Vert_{L^2(\omega)} \right)^{1-\kappa}
\end{equation} for all $\phi\in H^2(G_{T+\varepsilon})$ such that $\phi=0$ in $G\times \{0\}.$ }}
}\\

\subsubsection{ Construction of a suitable function $\phi$:}
Let us apply \eqref{Global:Carleman:EstimateProofSpectralIneq} with $\phi$ defined as follows. Consider $\psi\in C^\infty( G\times [0,T+\varepsilon])$ satisfying that
\begin{eqnarray}\psi(t):=
\begin{cases}C ,& \text{ }t\in [0,T],\,x\in G,
\\
0,& \text{ } t\in [T+\frac{3\varepsilon}{4},T+\varepsilon],\,x\in G,
\end{cases}
\end{eqnarray} where $0<C\leq \varepsilon. $ Assume that there exists $M_0>0,$ independent of $\varepsilon$ such that
\begin{equation}\label{BoundedNorm;phi:i}
    \Vert \psi^{(i)}\Vert_{L^\infty}\leq M_0,
\end{equation}for all $i\in \{1,2,3,4\}.$ We construct a function with this characteristics  in Lemma \ref{Lemma:fucntion:psi}. Then, 
by considering the function
\begin{equation}
    F(x,t):=\sum_{\lambda_j\leq \lambda}\frac{\sinh(\lambda_jt)}{\lambda_j} a_je_j(x),\,\,(x,t)\in G_T:=G\times [0,T],
\end{equation} and its extension to the  set $I_{T+\varepsilon}=[0,T+\varepsilon]$ by the constant function  equal to $F(x,T)$, that is
\begin{eqnarray}F(x,t)=
\begin{cases}\sum_{\lambda_j\leq \lambda}\frac{\sinh(\lambda_jt)}{\lambda_j} a_je_j(x),\,\,(x,t) ,& \text{ }t\in [0,T],\,x\in G,
\\
F(x,T),& \text{ } t\in [T,T+\varepsilon],\,x\in G,
\end{cases}
\end{eqnarray}we consider 
\begin{equation}\label{Auxialiar:phi:}
    \phi(x,t):=F(x,t)\psi(t),\,\,(x,t)\in G_{T+\varepsilon}.
\end{equation}Note that $\phi(x,0)=F(x,0)=0=\phi(x,T+\varepsilon)=0.$ Note also that $\phi$ is an extension  of $F$ from $G_T$ to $G_{T+\varepsilon}.$ 

Now, let us consider the odd extension of $\phi$ to the whole interval $[-(T+\varepsilon),T+\varepsilon],$ that is,
\begin{equation}
    \phi(x,t)=-\phi(x,-t)=-\psi(-t)F(x,-t),\, -(T+\varepsilon)\,\leq t\leq 0.
\end{equation}

\subsubsection{The norm $\Vert \partial_t\phi(x,0)\Vert_{L^2(\omega)}$ } Note that $\psi$ has been defined on $[0,T+\varepsilon]$ and it has been extended to $[-(T+\varepsilon),0]$ using its odd extension, that is, the one defined via  $\psi(-t)=-\psi(t),$ for $t\in [0,T+\varepsilon].$

The Leibniz rule  gives for any $t$ in a neighborhood of $t=0,$ the identity
$$  \partial_t\phi(x,t)=\psi'(t)F(x,t)+\psi(t)\partial_tF(x,t). $$ By evaluation both sides of this identity at $t=0$ we have
\begin{equation}
    \partial_t\phi(x,0)=\psi'(0)F(x,0)+\psi(0)\partial_tF(x,0)=\psi(0)\partial_tF(x,0),
\end{equation} from which we have the identity
\begin{equation}
    \Vert \partial_t\phi(x,0)\Vert_{L^2(\omega)}=|\psi(0)|\times \Vert\partial_t F(x,0) \Vert_{L^2(\omega)}=|\psi(0)|\times\Vert\varkappa \Vert_{L^2(\omega)}.
\end{equation}

\subsubsection{Embedding of $G_T$ in a closed manifold $G\times \mathbb{T}({T,\varepsilon}),$  $\mathbb{T}({T,\varepsilon})\cong \mathbb{S}^1$:} Now, we will proceed with a topological construction. 

It is clear that in the variable $t\in [-(T+\varepsilon),T+\varepsilon]$ the function $\phi$ can be extended in the periodic way to the whole line $\mathbb{R},$ or in other words, we can identify $\phi$ with a distribution on the dilated torus
\begin{equation}
    \mathbb{T}({T,\varepsilon})=\mathbb{R}/(2(T+\varepsilon)\mathbb{Z})=[-(T+\varepsilon),T+\varepsilon],
\end{equation}where in the resulting manifold $[-(T+\varepsilon),T+\varepsilon]$ we identify the endpoints  $-(T+\varepsilon)\sim T+\varepsilon.$ This construction allows the manifold $\mathbb{T}({T,\varepsilon})$ to be diffeomorphic to the circle $\mathbb{S}^1,$ and in consequence the function $\phi\in \mathscr{D}'(G\times \mathbb{T}({T,\varepsilon}))$ is  smooth on the product space  $G\times \mathbb{T}({T,\varepsilon})$ which is a compact manifold of $C^\infty$-class without boundary. In particular, we have the inclusion:
$\forall T,\varepsilon>0,\, G_T\subset G\times \mathbb{T}({T,\varepsilon}). $
 
\subsubsection{Proof of the interpolation inequality} Now, we are ready for the proof of the interpolation inequality \eqref{Interpolation:Inequality}
$$ \Vert F \Vert_{H^1(G\times (\alpha,T-\alpha))}\leq C_{s_0,s_{00}}\Vert F\Vert_{H^1(G_{T})}^\kappa \Vert \varkappa \Vert_{L^2(\omega)}^{1-\kappa}.$$
In the identity \eqref{symmetry:l2:equation:Lapla} below, we will prove that with $\phi$ defined in \eqref{Auxialiar:phi:}, we have that
\begin{align*}
    \Vert (-\partial_t^2+\mathcal{L}_G)\phi(x,t) \Vert_{L^2(G_{T+\varepsilon})}= 1/\sqrt{2}  \Vert (-\partial_t^2+\mathcal{L}_G)\phi(x,t) \Vert_{L^2(G\times \mathbb{T}({T,\varepsilon}))}.
\end{align*}
The positivity condition
$$    A\geq cI,\,\,\,c>0,$$ 
and the spectral mapping theorem applied to $A$ gives the lower bound
$$    A^{\frac{2}{  m  }}\geq c^{\frac{2}{  m  }}I,\,\,\,c>0,$$ 
which
gives the invertibility of the pseudo-differential operator $$ \mathcal{A}(t,x,D,\partial_t)=-\partial_t^2+A^{\frac{2}{  m  }}:H^2(G\times \mathbb{T}({T,\varepsilon}))\rightarrow L^2(G\times \mathbb{T}({T,\varepsilon})).$$ Then the operator
$$ \mathcal{A}(t,x,D,\partial_t)^{-1}:L^2(G\times \mathbb{T}({T,\varepsilon}))\rightarrow H^2(G\times \mathbb{T}({T,\varepsilon}))$$  is bounded, see Lemma \ref{The:constant:B}.
Moreover,
\begin{align*}
      \Vert (-\partial_t^2+\mathcal{L}_G)\phi(x,t) \Vert_{L^2(G\times \mathbb{T}({T,\varepsilon}))}\leq C \Vert(-\partial_t^2+A^{\frac{2}{  m  }})\phi(x,t)\Vert_{L^2(G\times \mathbb{T}({T,\varepsilon}))},
\end{align*}where the constant $C$ is independent of $\varepsilon>0.$ Now, let us use the identity 
$$ 
     \Vert (-\partial_t^2+\mathcal{L}_G)\phi(x,t) \Vert_{L^2(G\times \mathbb{T}({T,\varepsilon}))}$$
     $$= \Vert (-\partial_t^2+\mathcal{L}_G)(-\partial_t^2+A^{\frac{2}{m}})^{-1} (-\partial_t^2+A^{\frac{2}{m}})\phi(x,t) \Vert_{L^2(G\times \mathbb{T}({T,\varepsilon}))}.
$$ From Lemma \ref{Finite:Constant:CV}, the operator $ (-\partial_t^2+\mathcal{L}_G) (-\partial_t^2+A^{\frac{2}{m}})^{-1} $ belongs to the class $\Psi^0_{1,\delta}$ and the matrix-valued  Calder\'on-Vaillancourt theorem gives its boundedness on $L^2,$ with its operator norm  bounded by a constant $C>0,$ independent of $\varepsilon>0.$
Consequently, 
$$ \Vert (-\partial_t^2+\mathcal{L}_G)(-\partial_t^2+A^{\frac{2}{m}})^{-1} (-\partial_t^2+A^{\frac{2}{m}})\phi(x,t) \Vert_{L^2(G\times \mathbb{T}({T,\varepsilon}))}$$ 
$$\leq C \Vert  (-\partial_t^2+A^{\frac{2}{m}})\phi(x,t) \Vert_{L^2(G\times \mathbb{T}({T,\varepsilon}))}.
$$
In what follows we estimate the norm: 
\begin{equation*}
    Z_1:=   \Vert(-\partial_t^2+A^{\frac{2}{  m  }})\phi(x,t)\Vert_{L^2(G\times \mathbb{T}({T,\varepsilon}))},
\end{equation*}
and let us keep in mind that the partial analysis above gives us the inequality
\begin{equation}\label{Carleman:partial:phi}
  \Vert \phi \Vert_{H^1(G\times (\alpha,T-\alpha))}\lesssim \Vert \phi\Vert_{H^1(G_{T+\varepsilon})}^\kappa(Z_1+|\psi(0)|\Vert \varkappa\Vert_{L^2(\omega)})^{1-\kappa} ,  
\end{equation}
in view of the identity $$\psi(0)\varkappa=\partial_t\phi(x,0)=\psi(0)\partial_tF(x,0)=\psi(0)\sum_{\lambda_j\leq \lambda} a_je_j(x).$$ Indeed, we recall that
 for every $t\in [0,T],$ we have that
$$  \partial_t F(x,t)=\sum_{\lambda_j\leq \lambda}{\sinh(\lambda_jt)} a_je_j(x),$$
$$  \partial_t^2 F(x,t)=\sum_{\lambda_j\leq \lambda}{\sinh(\lambda_jt)}\lambda_j a_je_j(x).$$

\subsubsection{Estimate of $Z_1$} By the spectral properties of $A$ we have 
$$ A^{\frac{2}{  m  }}F(x,t)= \sum_{\lambda_j\leq \lambda}\frac{\sinh(\lambda_jt)}{\lambda_j} a_j A^{\frac{2}{  m  }}(e_j)(x)=\sum_{\lambda_j\leq \lambda}\frac{\sinh(\lambda_jt)}{\lambda_j} a_j \lambda_j^2e_j(x) $$
$$ =\partial_t^2 F(x,t),  $$ for all $t\in [0,T].$
Since $\phi(x,t)=\psi(t)F(x,t)$ on $[0, T],$ and $\psi$ is constant on $[0,T]$ we have that 
\begin{equation}\label{Cancellation}
    \forall x\in G,\,\forall t\in (0,T),\, (-\partial_t^2+A^{\frac{2}{  m  }})\phi(x,t)=0.
\end{equation}
First, note that the following symmetry property is valid due to the identity $\phi(x,t)=-\phi(x,-t),$

\begin{equation}\label{symmetry:l2:equation}
   \Vert (-\partial_t^2+A^{\frac{2}{  m  }})\phi(x,t) \Vert_{L^2(G\times \mathbb{T}({T,\varepsilon}))}^2= 2\Vert (-\partial_t^2+A^{\frac{2}{  m  }})\phi(x,t) \Vert^2_{L^2(G_{T+\varepsilon})}.
\end{equation}
Moreover, once proved \eqref{symmetry:l2:equation}, we have in the particular case (where $A=+\mathcal{L}_G$) of the Laplacian, the following inequality.
\begin{equation}\label{symmetry:l2:equation:Lapla}
    \Vert (-\partial_t^2+\mathcal{L}_G)\phi(x,t) \Vert_{L^2(G\times \mathbb{T}({T,\varepsilon}))}^2= 2\Vert (-\partial_t^2+\mathcal{L}_G)\phi(x,t) \Vert^2_{L^2(G_{T+\varepsilon})}.
\end{equation}
Indeed, for the proof of the identity of norms in  \eqref{symmetry:l2:equation} observe that
\begin{align*}
    & \Vert (-\partial_t^2+A^{\frac{2}{  m  }})\phi(x,t) \Vert_{L^2(G\times \mathbb{T}({T,\varepsilon}))}^2\\
     &=\int\limits_{G}\int\limits_{-T-\varepsilon}^0\vert (-\partial_t^2+A^{\frac{2}{  m  }})\phi(x,t) \vert^2 dt\,dx+\int\limits_{G}\int\limits_{0}^{T+\varepsilon}\vert (-\partial_t^2+A^{\frac{2}{  m  }})\phi(x,t) \vert^2 dt\,dx\\
     &=\int\limits_{G}\int\limits_{-T-\varepsilon}^0\vert (-\partial_t^2+A^{\frac{2}{  m  }})(-\phi(x,-t) \vert^2 dt\,dx+\int\limits_{G}\int\limits_{0}^{T+\varepsilon}\vert (-\partial_t^2+A^{\frac{2}{  m  }})\phi(x,t) \vert^2 dt\,dx\\
     &=\int\limits_{G}\int\limits_{-T-\varepsilon}^0\vert \phi_{tt}(x,-t)-A^{\frac{2}{  m  }}\phi(x,-t) \vert^2 dt\,dx+\int\limits_{G}\int\limits_{0}^{T+\varepsilon}\vert (-\partial_t^2+A^{\frac{2}{  m  }})\phi(x,t) \vert^2 dt\,dx\\
     &=\int\limits_{G}\int\limits_{0}^{T+\varepsilon}\vert -\phi_{tt}(x,t)+A^{\frac{2}{  m  }}\phi(x,t) \vert^2 dt\,dx+\int\limits_{G}\int\limits_{0}^{T+\varepsilon}\vert (-\partial_t^2+A^{\frac{2}{  m  }})\phi(x,t) \vert^2 dt\,dx\\
&=2\int\limits_{G}\int\limits_{0}^{T+\varepsilon}\vert (-\partial_t^2+A^{\frac{2}{  m  }})\phi(x,t) \vert^2 dt\,dx.
\end{align*}

Taking into account \eqref{Cancellation}, the symmetry property  $\phi(x,t)=-\phi(x,-t),$   and the positivity of the operator $(-\partial_t^2+A^{\frac{2}{  m  }})$ on $L^2(G\times \mathbb{T}({T,\varepsilon}))$ (that is, making use of the self-adjointness of $(-\partial_t^2+A^{\frac{2}{  m  }})$)  imply that
\begin{align*}
   \Vert (-\partial_t^2+A^{\frac{2}{  m  }}) &\phi(x,t) \Vert_{L^2(G_{T+\varepsilon})}^2=\frac{1}{2}\Vert (-\partial_t^2+A^{\frac{2}{  m  }})\phi(x,t) \Vert_{L^2(G\times \mathbb{T}({T,\varepsilon}))}^2\\
   &=\frac{1}{2}|((-\partial_t^2+A^{\frac{2}{  m  }})\phi),(-\partial_t^2+A^{\frac{2}{  m  }})\phi)_{L^2(G\times \mathbb{T}({T,\varepsilon}))}|\\
   &=\frac{1}{2}|((-\partial_t^2+A^{\frac{2}{  m  }})^2\phi),\phi)_{L^2(G\times \mathbb{T}({T,\varepsilon}))}|\\
   &\leq \int\limits_{G}\int\limits_{ [0,T+\varepsilon)}|(-\partial_t^2+A^{\frac{2}{  m  }})^2\phi(x,t)||\overline{\phi(x,t)}|dxdt\\
   &=\int\limits_{G}\int\limits_{ [T,T+\varepsilon)}|(-\partial_t^2+A^{\frac{2}{  m  }})^2\phi(x,t)||{\phi(x,t)}|dxdt.
\end{align*}Therefore, we have the estimate
\begin{align*}
   \Vert (-\partial_t^2+A^{\frac{2}{  m  }})\phi(x,t) \Vert_{L^2(G_{T+\varepsilon})}^2
   &\leq \int\limits_{G}\int\limits_{ [T,T+\varepsilon)}|(-\partial_t^2+A^{\frac{2}{  m  }})^2\phi(x,t)||\phi(x,t)|dxdt\\
   &\leq \int\limits_{G}\int\limits_{ [T,T+\varepsilon)}|\phi(x,t)|dxdt\times \|(-\partial_t^2+A^{\frac{2}{  m  }})^2\phi\|_{L^\infty}\\
   &=I\times II,
\end{align*}where 
\begin{equation*}
    I=\int\limits_{G}\int\limits_{ [T,T+\varepsilon)}|\phi(x,t)|dx\,dt,\,II= \|(-\partial_t^2+A^{\frac{2}{  m  }})^2\phi\|_{L^\infty(G\times [T,T+\varepsilon))}.
\end{equation*} Now, we will estimate each one of these norms.\\

\subsubsection{Estimate for $I$} Note that
\begin{align*}
    I\leq \textnormal{Vol}(G)\times \varepsilon\times \Vert \phi\Vert_{L^\infty(G\times [T,T+\varepsilon])}=\textnormal{Vol}(G)\times \varepsilon \Vert \psi(t)\Vert_{L^\infty[T,T+\varepsilon]} \Vert F(x,t)\Vert_{L^\infty}.
\end{align*}Now, and  for any $t$ fixed, and for all  $s_0\in \mathbb{N}$ observe that
\begin{equation}
    |F(x,t)|\leq \sup_{0\leq s\leq s_0} \|(1+A)^{\frac{s}{  m  }}F(\cdot,t)\|_{L^\infty(G)}.
\end{equation}In view of  the Sobolev embedding theorem, any $s_{00}>n/2$ satisfies that
\begin{align*}
  \sup_{0\leq s\leq s_0} \|(1+A)^{\frac{s}{  m  }}F(\cdot,t)\|_{L^\infty(G)}&\leq \sup_{0\leq s\leq s_0} \|(1+A)^{\frac{s+s_{00}}{  m  }}F(\cdot,t)\|_{L^2(G)}\\
  &\leq \sup_{0\leq s\leq s_0+s_{00}} \|(1+A)^{\frac{s}{  m  }}F(\cdot,t)\|_{L^2(G)}.
\end{align*}The spectral properties  of the operator $(1+A)^{\frac{s}{  m  }}$ give
the estimates
\begin{align*}
   \|(1+A)^{\frac{s}{  m  }}F(\cdot,t)\|_{L^2(G)}^2 &=\left\Vert  \sum_{\lambda_j\leq \lambda}\frac{\sinh(\lambda_jt)}{\lambda_j}(1+\lambda_j^  m  )^{\frac{s}{  m  }}a_je_j(x) \right\Vert^2_{L^2(G)} \\
   &=\sum_{\lambda_j\leq \lambda}\left|\frac{\sinh(\lambda_jt)}{\lambda_j}\right|^2(1+\lambda_j^  m  )^{\frac{2s}{  m  }}|a_j|^2\\
   &\lesssim \sum_{\lambda_j\leq \lambda}\left|\frac{\sinh(\lambda_jt)}{\lambda_j}\right|^2\lambda_j^{2s}|a_j|^2\lesssim \sum_{\lambda_j\leq \lambda}e^{\lambda_jt}\lambda_j^{2s-2}|a_j|^2\\
   &\lesssim_{s_0,s_{00}} e^{C\lambda T} \sum_{\lambda_j\leq \lambda}|a_j|^2\\
   &= e^{C\lambda T} \Vert \partial_t F(\cdot,0)\Vert_{L^2(G)}^2,
\end{align*}for some $C>1.$
So,  we deduce the inequality
\begin{equation}\label{auxialiar:s00}
    \forall s\in [0,s_0],\,\forall s_{00}>n/2,\,\,\Vert (1+A)^{\frac{s}{  m  }} F\Vert_{L^\infty}\lesssim_{s_0,s_{00}}e^{C\lambda T/2}\|\partial_t F(\cdot,0)\Vert_{L^2(G)},
\end{equation}
as well as the Sobolev inequality
\begin{equation}\label{auxialiar:s0}
  \forall s_{00}>n/2,\, \forall s\in [0,s_0+s_{00}],\,\,\,\Vert (1+A)^{\frac{s}{  m  }} F\Vert_{L^2}\lesssim_{s_0,s_{00}}e^{C\lambda T/2}\|\partial_t F(\cdot,0)\Vert_{L^2(G)}.
\end{equation}
With $s_0=0,$ we have that $\Vert  F\Vert_{L^\infty}\lesssim_{s_0,s_{00}}e^{ T\lambda /2}\|\partial_t F(\cdot,0)\Vert_{L^2(G)}.$ Putting all these estimates together we have the inequality:
\begin{equation}\label{Estimate:I}
    I\lesssim \textnormal{Vol}(G)\times \varepsilon \Vert \psi(t)\Vert_{L^\infty[T,T+\varepsilon]} e^{C\lambda T/2} \|\partial_t F(\cdot,0)\Vert_{L^2(G)}=\textnormal{Vol}(G)\times \varepsilon \Vert \psi(t)\Vert_{L^\infty[T,T+\varepsilon]} e^{C\lambda T/2}.
\end{equation} In the last line we have used the identity  $\|\partial_t F(\cdot,0)\Vert_{L^2(G)}=\Vert \varkappa\Vert_{L^2(G)}=1.$ Summarising, we have the inequality 
 $$   I\leq C' \textnormal{Vol}(G)\times \varepsilon \Vert \psi(t)\Vert_{L^\infty[T,T+\varepsilon]} e^{C\lambda T/2}, $$ for some $C'>0$ independent of $\varepsilon >0.$

\subsubsection{Estimating  $II$:} To estimate the second term, we start by observing the inequality
\begin{equation*}
    II= \|(-\partial_t^2+A^{\frac{2}{  m  }})^2\phi\|_{L^\infty(G\times [T,T+\varepsilon))}=\|(-\partial_t^2+A^{\frac{2}{  m  }})^2[\psi(t)F(x,T)]\|_{L^\infty(G\times [T,T+\varepsilon))}.
\end{equation*}Since
\begin{align*}
   (-\partial_t^2+ A^{\frac{2}{  m  }})^2 &[\psi(t)F(x,T)]=(-\partial_t^2+ A^{\frac{2}{  m  }})(-\partial_t^2+ A^{\frac{2}{  m  }})[\psi(t)F(x,T)]\\
   &=(-\partial_t^2+A^{\frac{2}{  m  }})[-\psi_{tt}(t)F(x,T)+\psi(t)A^{\frac{2}{  m  }}F(x,T)]\\
   &=\psi^{(4)}(t)F(x,T)-2\psi_{tt}(t)A^\frac{2}{  m  }F(x,T)\\
   &\,\,\,\,+\psi(t)A^{\frac{4}{  m  }}(F(x,T)),
\end{align*}for $s_0\geq 4,$ and with $s_{00}>n/2,$  the Sobolev inequality in \eqref{auxialiar:s00} implies that
\begin{align*}
  &  \|(-\partial_t^2+A^{\frac{2}{  m  }})^2[\psi(t)F(x,T)]\|_{L^\infty(G\times [T,T+\varepsilon))}\\
  &\leq \Vert \psi^{(4)}\Vert_{L^\infty[T,T+\varepsilon]}\Vert F(x,T)\Vert_{L^\infty}\\
  &+2\Vert\psi_{tt} \Vert_{L^\infty[T,T+\varepsilon]}\Vert A^{\frac{2}{  m  }}F(x,T) \Vert_{L^\infty}+\Vert\psi\Vert_{{L^\infty[T,T+\varepsilon]}}\Vert A^{\frac{4}{  m  }}F(x,T) \Vert_{L^\infty}\\
  &\leq \Vert \psi^{(4)}\Vert_{L^\infty[T,T+\varepsilon]}\Vert F(x,T)\Vert_{L^\infty}\\
  &+2\Vert\psi_{tt} \Vert_{L^\infty[T,T+\varepsilon]}\Vert(1+ A)^{\frac{s_{00}+2}{  m  }}F(x,T) \Vert_{L^2}\\
  &+\Vert\psi\Vert_{{L^\infty[T,T+\varepsilon]}}\Vert (1+ A)^{\frac{s_{00}+4}{  m  }}F(x,T) \Vert_{L^2}\\
  & \lesssim_{s_0,s_{00}}e^{{T} {\lambda}/2 }\|\partial_t F(\cdot,0)\Vert_{L^2(G)} (\Vert \psi^{(4)}\Vert_{L^\infty[T,T+\varepsilon]}+2\Vert\psi_{tt} \Vert_{L^\infty[T,T+\varepsilon]}+\Vert\psi\Vert_{{L^\infty[T,T+\varepsilon]}})\\
 &=e^{{T} {\lambda}/2 } (\Vert \psi^{(4)}\Vert_{L^\infty[T,T+\varepsilon]}+2\Vert\psi_{tt} \Vert_{L^\infty[T,T+\varepsilon]}+\Vert\psi\Vert_{{L^\infty[T,T+\varepsilon]}}).
\end{align*}
\subsubsection{Estimate for $Z_1$}
In view of the estimates for $I$ and  $II$ above,  we have that
\begin{align*}
  &  \Vert (-\partial_t^2+E(x,D)^{\frac{2}{\nu}})\phi(x,t) \Vert_{L^2(G_{T+\varepsilon})}^2\\
    &\lesssim \textnormal{Vol}(G)\times \varepsilon \Vert \psi(t)\Vert_{L^\infty[T,T+\varepsilon]} e^{T\lambda}  (\Vert \psi^{(4)}\Vert_{L^\infty[T,T+\varepsilon]}+2\Vert\psi_{tt} \Vert_{L^\infty[T,T+\varepsilon]}+\Vert\psi\Vert_{{L^\infty[T,T+\varepsilon]}}).
\end{align*}

\subsubsection{Final Analysis}
The estimates above for $Z_1$  lead to the following inequality in view of the interpolation inequality
 \eqref{Carleman:partial:phi}
\begin{align*}
 &\Vert \phi\Vert_{H^1(G\times (\alpha,T-\alpha))}\\
 &\lesssim \Vert \phi\Vert_{H^1(G_{T+\varepsilon})}^\kappa((\textnormal{Vol}(G)\times \varepsilon \Vert \psi(t)\Vert_{L^\infty[T,T+\varepsilon]} e^{T\lambda} (\Vert \psi^{(4)}\Vert_{L^\infty[T,T+\varepsilon]}+2\Vert\psi_{tt} \Vert_{L^\infty[T,T+\varepsilon]}+\Vert\psi\Vert_{{L^\infty[T,T+\varepsilon]}}))^{\frac{1}{2}}\\
 &+|\psi(0)|\Vert \varkappa \Vert_{L^2(\omega)})^{1-\kappa} .
 \end{align*} 
Now,  dividing both sides of this inequality by $|\phi(0)|$ and using that $\psi(0)=\psi(T)=\psi(t),$ $0\leq t\leq T,$ we get 
\begin{align*}
 & \Vert F(x,t)\Vert_{H^1(G\times (\alpha,T-\alpha))}=\left\Vert \frac{\psi(t) F(x,t)}{\psi(0)}\right\Vert_{H^1(G\times (\alpha,T-\alpha))}\\
 &\lesssim \left\Vert \frac{\psi(t)F(x,t)}{\psi(0)}\right\Vert_{H^1(G_{T+\varepsilon})}^\kappa\\
 &\times\frac{1}{|\psi(0)|^{1-\kappa}} (  (\textnormal{Vol}(G)\times \varepsilon \Vert \psi(t)\Vert_{L^\infty[T,T+\varepsilon]} e^{T\lambda} (\Vert \psi^{(4)}\Vert_{L^\infty[T,T+\varepsilon]}+2\Vert\psi_{tt} \Vert_{L^\infty[T,T+\varepsilon]}+\Vert\psi\Vert_{{L^\infty[T,T+\varepsilon]}}))^{\frac{1}{2}}\\
 &+|\psi(0)|\Vert \varkappa \Vert_{L^2(\omega)})^{1-\kappa} \\
 &=\left\Vert \frac{\psi(t)F(x,t)}{\psi(0)}\right\Vert_{H^1(G_{T+\varepsilon})}^\kappa\\
 &\times (\left(\textnormal{Vol}(G)\times \varepsilon  e^{T\lambda}\frac{\Vert \psi(t)\Vert_{L^\infty[T,T+\varepsilon]}}{|\psi(0)|} \left(\frac{\Vert \psi^{(4)}\Vert_{L^\infty[T,T+\varepsilon]}}{\| \psi(0)|}+\frac{2\Vert\psi_{tt} \Vert_{L^\infty[T,T+\varepsilon]}}{| \psi(0)|}+1\right)\right)^{\frac{1}{2}}\\
 &+\Vert \varkappa \Vert_{L^2(\omega)})^{1-\kappa}\\
 &=\left\Vert \frac{\psi(t)F(x,t)}{\psi(T)}\right\Vert_{H^1(G_{T+\varepsilon})}^\kappa\\
 &\times (\left(\textnormal{Vol}(G)\times \varepsilon  e^{T\lambda}\frac{\Vert \psi(t)\Vert_{L^\infty[T,T+\varepsilon]}}{|\psi(T)|} \left(\frac{\Vert \psi^{(4)}\Vert_{L^\infty[T,T+\varepsilon]}}{|\psi(T)|}+\frac{2\Vert\psi_{tt} \Vert_{L^\infty[T,T+\varepsilon]}}{|\psi(T)|}+1\right)\right)^{\frac{1}{2}}\\
 &+\Vert \varkappa \Vert_{L^2(\omega)})^{1-\kappa}. 
 \end{align*} 
 and taking the limit when $\varepsilon\rightarrow 0^+$ in both sides of this estimate
we  conclude the expected inequality,
\begin{equation}\label{Expected:Inequality}
    \Vert F \Vert_{H^1(G\times (\alpha,T-\alpha))}\leq C_{s_0,s_{00}}\Vert F\Vert_{H^1(G_{T})}^\kappa \Vert \varkappa \Vert_{L^2(\omega)}^{1-\kappa},
\end{equation}
where we have used  that $\psi(t)/\psi(0)=\psi(t)/\psi(T)=1,$ $0\leq t\leq T,$ the smoothness of $\psi,$ the following identities (see proposition \ref{Lemma:fucntion:psi}) $$ \lim_{\varepsilon\rightarrow0^+}\Vert \psi^{(4)}\Vert_{L^\infty[T,T+\varepsilon]}=\psi^{(4)}(T)=\lim_{\varepsilon\rightarrow0^+}\Vert \psi_{tt}\Vert_{L^\infty[T,T+\varepsilon]}=\psi_{tt}(T)=0,$$ and the following facts

    \begin{equation}
        \lim_{
    \varepsilon\rightarrow 0^+
    }\left\Vert{\psi(t)}/{\psi(T)} \right\Vert_{L^\infty[T,T+\varepsilon]}=1,
    \end{equation}and
    \item 
    \begin{equation}\label{2:33}
 \lim_{\varepsilon\rightarrow 0}\left\Vert{\psi(t)F(x,t)}/{\psi(0)}\right\Vert_{H^1(G_{T+\varepsilon})}= \Vert F(x,t)\Vert_{H^1(G_{T})}.  
    \end{equation}
For the proof of \eqref{2:33} note that for $\tilde{\psi}:=\psi(t)/\psi(0),$ and using that $F(x,t)=F(x,T)$ if $0\leq t\leq T+\varepsilon,$ we have
$$ 
  \lim_{\varepsilon\rightarrow 0}\left\Vert{\psi(t)F(x,t)}/{\psi(0)}\right\Vert^2_{H^1(G_{T+\varepsilon})}=\lim_{\varepsilon\rightarrow 0}\sum_{j=0,1}\int\limits_{0}^{T+\varepsilon}\Vert\partial_t^{(j)}(\tilde{\psi}(t)F(x,t))\Vert^2_{H^1(G)}dt $$ 
  $$ 
  =\lim_{\varepsilon\rightarrow 0}\sum_{j=0,1}\int\limits_{0}^{T}\Vert\partial_t^{(j)}(\tilde{\psi}(t)F(x,t))\Vert^2_{H^1(G)}dt+\lim_{\varepsilon\rightarrow 0}\sum_{j=0,1}\int\limits_{T}^{T+\varepsilon}\Vert\partial_t^{(j)}(\tilde{\psi}(t)F(x,t))\Vert^2_{H^1(G)}dt $$ 
  $$ 
  =\sum_{j=0,1}\int\limits_{0}^{T}\Vert\partial_t^{(j)}(F(x,t))\Vert^2_{H^1(G)}dt+\lim_{\varepsilon\rightarrow 0}\sum_{j=0,1}\int\limits_{T}^{T+\varepsilon}\Vert\tilde{\psi}^{(j)}(t)F(x,T)\Vert^2_{H^1(G)}dt $$   $$= \Vert F(x,t)\Vert_{H^1(G_{T})}^2, $$
where we have used that when $t\rightarrow T, $ $\tilde{\psi}^{(j)}(t)\rightarrow 0.$
Now, from \eqref{Expected:Inequality} we can follow the standard Lebeau-Robbiano argument that has been described at the beginning of the section to conclude the proof of the spectral inequality.  Having proved \eqref{ObservabilityInequality}, the proof  of Proposition \ref{Lemma:LR;Ineq} is complete.
\end{proof}

\begin{proof}[Proof of Theorem \ref{Main:theorem}] Let $\mu,c>0$ be two positive parameters and assume that $0<c<\mu^  m  $. Define
$$\tilde{E}(x,D)=A+c.$$ Note that $\tilde{E}(x,D)\geq cI.$ Since $a(x,[\xi])\geq 0,$ the global symbol 
$$ \tilde{E}(x,[\xi])=a(x,[\xi])+cI_{d_\xi}  ,$$ satisfies the positivity condition $$\forall[\xi]\in \widehat{G},\,\tilde{E}(x,[\xi])\geq c I_{d_\xi}.$$ On the other hand, 
if  $\{\mu_j:=\lambda_j^  m  ,e_j\}$ are the corresponding spectral data $$ Ae_j=\lambda_j^  m   e_j,\,\,\lambda_j\geq 0,$$ of $A,$ then  $\{\mu_j+c:=\lambda_j^  m  +c,e_j\}$ are the corresponding spectral data $$\tilde{ E}(x,D)e_j=(\lambda_j^  m  +c) e_j,\,\,\lambda_j\geq 0$$  of the operator $\tilde{E}(x,D).$  Let $\lambda:=(\mu^  m  +c)^{\frac{1}{  m  }}.$ From Proposition \ref{Lemma:LR;Ineq} we deduce the spectral inequality
\begin{equation}\label{ObservabilityInequality:Proof:II}
    \left(\sum_{ (\lambda_j^  m  +c)^{\frac{1}{  m  } }\leq \lambda}a_j^2\right)^\frac{1}{2}\leq C_1e^{C_2{\lambda}}\left\Vert \sum_{{ (\lambda_j^  m  +c)^{\frac{1}{  m  } }\leq \lambda}}a_je_j(x)  \right\Vert_{L^2(\omega)}.
\end{equation}Note that $(\lambda_j^  m  +c)^{\frac{1}{  m  } }\leq \lambda$ becomes equivalent to the inequality $\lambda_j\leq \mu$ and since $0<c<\mu^  m  ,$ then $\lambda<2\mu.$ Thus, we have proved the spectral inequality 
\begin{equation}
    \left(\sum_{\lambda_j\leq \mu}a_j^2\right)^\frac{1}{2}\leq C_1e^{ 2C_2{\mu}}\left\Vert \sum_{\lambda_j\leq \mu}a_je_j(x)  \right\Vert_{L^2(\omega)}.
\end{equation}In consequence the proof of \eqref{Spectral:Inequality:Intro} is complete. For the proof of \eqref{Donnelly-Fefferman} we can use \eqref{Spectral:Inequality:Intro} and the Sobolev embedding theorem.   Indeed, let $R>0$ and let us consider $s\in \mathbb{R}$ such that $s>n/2.$  For the proof of \eqref{Donnelly-Fefferman} we can use \eqref{Spectral:Inequality:Intro} and the Sobolev embedding theorem.   Indeed, let $R>0$ and let us consider $s\in \mathbb{R}$ such that $s>n/2.$ With $\omega=B(x,R)$ a ball  of radius $R>0$ we have that
\begin{equation}\label{Auxiliar:proof:1}
    \Vert \varkappa \Vert_{L^\infty(B(x,2R))}\leq  \Vert \varkappa \Vert_{L^\infty(G)}.
\end{equation} Now, the Sobolev embedding theorem and the inequality in  \eqref{Spectral:Inequality:Intro} imply that
\begin{align*}
    \Vert \varkappa\Vert_{L^\infty(G)}\lesssim \Vert (1+A)^{\frac{s}{m}}\varkappa \Vert_{L^2(G)}\lesssim (1+\lambda)^{s}\Vert \varkappa\Vert_{L^2(G)}
\end{align*}
$$   \lesssim(1+\lambda)^{s}C_{1,R}e^{C_{2,R}\lambda}\Vert \varkappa\Vert_{L^2(B(x,R))}.$$
By using \eqref{Auxiliar:proof:1} and  \eqref{Spectral:Inequality:Intro}  we conclude this analysis with the inequality
\begin{align*}
    \Vert \varkappa\Vert_{L^\infty(B(x,2R))}\lesssim  \Vert \varkappa\Vert_{L^\infty(G)}\leq e^{C_{2,R}'+C_{1,R}' \lambda}\Vert \varkappa\Vert_{L^2(B(x,2R))}\leq e^{C_{2,R}'+C_{1,R}' \lambda}\Vert \varkappa\Vert_{L^\infty(B(x,2R))},
\end{align*} for some $C_{1,R}'>C_{1,R}$ and  $C_{2,R}'>C_{2,R}.$  The proof of  Theorem \ref{Main:theorem} is complete.
\end{proof}

\subsection{Applications to control theory: Null-controllability for diffusion models}\label{Applications:contro:theory} Now, we give a consequence of Theorem \ref{Main:theorem} which we present in the following way.
\begin{theorem}\label{Main:theorem:statement} Let $A$ be a positive and elliptic pseudo-differential operator of order $  m  >0$ in the H\"ormander class $\Psi^  m  _{\rho,\delta}(G\times \widehat{G})$ and let $u_0\in L^2(G)$ be an initial datum.

Then, for any $\alpha>1/  m  ,$  the fractional diffusion model
\begin{equation}\label{Main:statement}
\begin{cases}u_t(x,t)+ A^\alpha u(x,t)=g(x,t)\cdot 1_\omega (x) ,& (x,t)\in G\times (0,T),
\\u(0,x)=u_0,\end{cases}
\end{equation} is null-controllable at any time $T>0,$ that is, there exists an input function $g=g(x,t)\in L^2(G)$ such that for any $x\in G,$ $u(x,T)=0.$ 
\end{theorem}
\begin{proof}
The spectral inequality in \eqref{ObservabilityInequality} allows us to make use of Theorem \ref{Miller:Theorem} with $\mathcal{A}=A^\alpha$ and with $B=S=M_{1_\omega}$ being the multiplication operator by the characteristic function $1_{\omega}$. Note that $M_{1_\omega}$ is bounded on $H=L^2(G).$ Observe that $\mathcal{A}^\gamma= A^{\alpha\gamma}=A^\frac{1}{  m  }$ satisfies \eqref{ObservabilityInequality} (that is, the inequality \eqref{Espectral:Inequality:Hilbert} holds) for $\alpha\gamma=1/  m  .$ Because $\gamma\in (0,1)$ if an only if $\alpha>1/  m  ,$ Theorem  \ref{Miller:Theorem} guarantees that this inequality on the fractional order $\alpha$ is a sufficient condition in order that \eqref{Main:statement} will be null-controllable in time $T>0.$ The proof of Theorem \ref{Main:theorem:statement} is complete.
\end{proof}

 In the following result we analyse the controllability cost of the model \eqref{Main:statement} when the time is small.
 \begin{corollary}\label{Cost:Control} The controllability cost $C_T$ for the fractional heat equation \eqref{Main:statement} over short times $T\in (0,1)$ satisfies 
 \begin{equation}
     C_T\leq C_1e^{C_2T^{-\beta}},
 \end{equation}where $\beta>1/(\alpha  m  -1).$
 \end{corollary}
 \begin{proof}
 For the proof, note that $A^\gamma= A^{\alpha\gamma}=A^\frac{1}{  m  },$ satisfies \eqref{ObservabilityInequality}  for $\alpha\gamma=1/  m  .$ Then, from Theorem  \ref{Miller:Theorem} we have the estimate $  C_T\leq C_1e^{C_2T^{-\beta}}$ for any $\beta>\gamma/(\gamma-1)=1/(\alpha  m  -1).$ The proof of Corollary \ref{Cost:Control} is complete.
 \end{proof}

\section{Appendix:  Construction of the cut-off function $\psi$}\label{Section:construction:psi} 
In this appendix we construct  the regularising function $\psi$ used in the proof of Proposition \ref{Lemma:LR;Ineq}.  For any $\varepsilon\in (0,1),$ let $a:= {3\varepsilon} /4.$ 
 We summarise the analysis above and some their straightforward consequences in the following lemma.
\begin{lemma}\label{Lemma:fucntion:psi}
The function $\psi$ as defined in \eqref{function:psi} satisfies the following properties. 
\begin{itemize}
    \item[A.] $0<\psi(0)<\varepsilon.$
    \item[B.] $\psi^{(i)}(T)=0,$ for $i\in {\{1,2,3,4\}  }.$
    \item[C.] For $i\in {1,2,3,4},$ $\psi^{(i)}\in C^{\infty}(0,T+\varepsilon),$ 
\end{itemize} and there is a constant $M_0>0,$ independent of $\varepsilon \in (0,1),$ such that
\begin{equation}
    \Vert\psi^{(i)} \Vert_{L^\infty}\leq M_0,
\end{equation}for all $i=1,2,3,4.$
\end{lemma}
\begin{proof}
We do this by the following steps. 
\begin{itemize}
    \item[Step 1.] Define the function
\begin{equation}
    \mathcal{E}(t)=\begin{cases}e^{-\frac{1}{a^2-t^2}}(a^2-t^2)^{10},& \text{ }t\in [0,a],
\\
0,& \text{ } t\in [a,\varepsilon].
\end{cases}
\end{equation}
\item[Step 2.]  By straightforward computation one can show that for any $t\in [0,a],$\\
\begin{itemize}
    \item[1.]  $\mathcal{E}_t(t)=2 t \exp(-1/(a^2-t^2)) (a^2-t^2)^8(-10 a^2+10 t^2-1).$
    \\
    \item[2.] $\mathcal{E}_{tt}(t)=-2 \exp(-1/(a^2-t^2)) (a^2-t^2)^6 (10 a^6+a^4 (1-210 t^2)+a^2 (390 t^4-38 t^2)-190 t^6+37 t^4-2 t^2).$
    \\
    \item[3.]  $\mathcal{E}^{(3)}(t)= 4 t \exp(-1/(a^2-t^2)) (a^2-t^2)^4 (270 a^8+a^6 (54-2520 t^2)+a^4 (5940 t^4-594 t^2+3)-54 a^2 (100 t^6-19 t^4+t^2)+t^2 (1710 t^6-486 t^4+51 t^2-2)) . $
    \\
    \item[4.] $ \mathcal{E}^{(4)}(t)=4 \exp(-1/(a^2-t^2)) (a^2-t^2)^2 (270 a^{12}-54 a^{10} (190 t^2-1)+3 a^8 (22470 t^4-954 t^2+1)-12 a^6 t^2 (14370 t^4-1581 t^2+25)+6 a^4 t^2 (35235 t^6-6714 t^4+371 t^2-2)-2 a^2 t^4 (62730 t^6-17415 t^4+1782 t^2-68)+t^4 (29070 t^8-10710 t^6+1635 t^4-124 t^2+4)).$
\end{itemize}
These explicit formulae, allow us to write the first fourth derivatives of $\psi$  in the form
\begin{equation}
    \mathcal{E}^{(i)}(t)=\exp(-1/(a^2-t^2))(a^2-t^2)^{10-2i }P_{i}(t,a),\,i\in \{1,2,3,4\}.
\end{equation}where the functions $P_i(t,a)\in \mathbb{C}[t,a]$ are polynomials in two variables. By  evaluating the functions $\mathcal{E}^{(i)},$ $i\in \{0,1,2,3,4\},$ at $t=0,$ we get
\vspace{0.1cm}
\begin{itemize}
    \item[5.] $\mathcal{E}(0)=a^{20} e^{-1/a^2}.$
    
    \item[6.] $\mathcal{E}'(0)=0.$
    
    \item[7.] $\mathcal{E}^{(3)}(0)=0.$
    
    \item[8.] $\mathcal{E}^{(4)}(0)=4a^{4}(270 a^{12}+54 a^{10}+3a^{8})e^{-1/a^2}.$
\end{itemize}
Now, consider the function
\begin{equation}
   \tilde{ \eta}(t):=\mathcal{E}(t)(1-\mathcal{B}t^2+\mathcal{C}t^4), \,\, 0\leq t\leq a,
\end{equation} where $\mathcal{B}$ and $\mathcal{C}$ are real parameters. Then, straightforward computation shows that when
\begin{itemize}
    \item[9.] $\mathcal{B}=(\mathcal{E}^{(2)}(0)-\mathcal{E}(0))/2;$
    \item[10.] $\mathcal{C}=(6(\mathcal{E}^{(2)}(0)-\mathcal{E}(0))\mathcal{E}^{(2)}(0)-\mathcal{E}^{(4)}(0) )/12\mathcal{E}(0),$
\end{itemize} the function $\tilde \eta$ satisfies the following properties
$$ \tilde \eta(0)=\mathcal{E}(0),\, \eta^{(i)}(0)=0,\,\,i=1,2,3,4. $$
Let $T>0.$ 
\end{itemize}  The analysis above shows that the function
\begin{eqnarray}\label{function:psi}\psi(t):=
\begin{cases}\mathcal{E}(0)\tilde{\eta}(t-T) ,& \text{ }t\in [T,T+a],
\\
\mathcal{E}(0){\tilde \eta}(0),& \text{ } t\in [0,T]
,
\\
0,& \text{ } t\in [T+a,T+\varepsilon],
\end{cases}
\end{eqnarray}satisfies the required properties of the lemma.  
\end{proof}

\bibliographystyle{amsplain}

\end{document}